\documentclass[letterpaper, 11pt,  reqno]{preprint}
\usepackage[utf8x]{inputenc}
\usepackage[T1]{fontenc}
\usepackage[english]{babel}
\usepackage[full]{textcomp}
\usepackage[osf]{newtxtext}

\usepackage{amsmath,amssymb,amscd,amsthm,amsxtra, esint, bbm, enumerate, relsize, xcolor}
\usepackage{mathtools} 
\usepackage{stmaryrd}

\usepackage{accents, cancel}
\usepackage{mathrsfs,bbm}

\usepackage{enumitem}
\usepackage{mhequ}
\usepackage{microtype}
\usepackage{cprotect}
\usepackage{margincomment}

\usepackage[implicit=true]{hyperref}
\usepackage{orcidlink}

\DeclarePairedDelimiter{\scal}{\langle}{\rangle}
\DeclarePairedDelimiter{\abs}{\lvert}{\rvert}

\makeatletter
\newcommand*{\bigcdot}{}
\DeclareRobustCommand*{\bigcdot}{%
  \mathord{\hspace{0.1em}\mathpalette\bigcdot@{}\hspace{0.1em}}%
}
\newcommand*{\bigcdot@scalefactor}{.5}
\newcommand*{\bigcdot@widthfactor}{1.15}
\newcommand*{\bigcdot@}[2]{%
  \sbox0{$#1\vcenter{}$}
  \sbox2{$#1\cdot\m@th$}%
  \hbox to \bigcdot@widthfactor\wd2{%
    \hfil
    \raise\ht0\hbox{%
      \scalebox{\bigcdot@scalefactor}{%
        \lower\ht0\hbox{$#1\bullet\m@th$}%
      }%
    }%
    \hfil
  }%
}
\makeatother

\allowdisplaybreaks[2]

\colorlet{darkblue}{blue!90!black}
\colorlet{darkgreen}{green!50!black}



\let\restr\upharpoonright

\newtheorem{theorem}{Theorem} [section]

\newtheorem{lemma}[theorem]{Lemma}
\newtheorem{proposition}[theorem]{Proposition}

\newtheorem{corollary}[theorem]{Corollary}

\theoremstyle{definition}
\newtheorem{remark}[theorem]{Remark}

\newtheorem{assumption}[theorem]{Assumption}

\DeclareMathOperator*{\supp}{supp}

\DeclareMathOperator{\Law}{Law}

\DeclareMathOperator*{\hsum}{\hat\sum}

\newcommand{\Z}{\mathbb{Z}}
\newcommand{\R}{\mathbb{R}}

\newcommand{\T}{\mathbb{T}}

\newcommand{\f}{\mathbf f}

\let \div \relax
\DeclareMathOperator{\div}{div}

\def\${\vert\!\vert\!\vert}

\newcommand{\bb}{\mathbb}

\let\P= \undefined
\newcommand{\P}{\mathbb{P}}

\newcommand{\E}{\mathbb{E}}

\def\one{\mathbf{1}}

\newcommand{\B}{\mathcal{B}}

\newcommand{\F}{\mathcal{F}}

\newcommand{\al}{\alpha}
\newcommand{\be}{\beta}
\newcommand{\dl}{\delta}

\newcommand{\ep}{\varepsilon}

\newcommand{\g}{\gamma}

\newcommand{\s}{\sigma}

\newcommand{\ft}{\widehat}

\newcommand{\wt}{\widetilde}

\newcommand{\ta}{\theta}

\newcommand{\les}{\lesssim}

\newcommand{\ind}{\mathbbm 1}

\renewcommand{\S}{\mathcal{S}}

\let\d\partial
\let\f\frac

\newcommand{\N}{\mathbb{N}}

\newcommand{\CF}{\mathcal{F}}
\newcommand{\CK}{\mathcal{K}}

\newcommand{\MM}{\mathcal{M}}
\newcommand{\CC}{\mathcal{C}}

\newcommand{\CP}{\mathcal{P}}

\newcommand{\eps}{\ep}

\renewcommand{\ft}{\widehat}

\numberwithin{equation}{section}
\numberwithin{theorem}{section}

\usepackage{tikz}

\makeatletter
\newcommand\RSloop{\@ifnextchar\bgroup\RSloopa\RSloopb}
\makeatother
\newcommand\RSloopa[1]{\bgroup\RSloop#1\relax\egroup\RSloop}
\newcommand\RSloopb[1]%
  {\ifx\relax#1%
   \else
     \ifcsname RS:#1\endcsname
       \csname RS:#1\endcsname
     \else
       \GenericError{(RS)}{RS Error: operator #1 undefined}{}{}%
     \fi
   \expandafter\RSloop
   \fi
  }
\newcommand\Xy{0}
\newcommand\RS[1]%
  {\begin{tikzpicture}
     [every node/.style=
       {circle,draw,fill,minimum size=2.5pt,inner sep=0pt,outer sep=0pt},
      line cap=round,baseline=0.5
     ]
   \coordinate(\Xy) at (0,0);
   \RSloop{#1}\relax
   \end{tikzpicture}
  }
\newcommand\RSdef[1]{\expandafter\def\csname RS:#1\endcsname}
\newlength\RSu
\RSdef{i}{\draw (\Xy) -- +(90:\RSu) node{};}
\RSdef{w}{\draw (\Xy) -- +(90:\RSu) node[fill=white]{};}
\RSdef{l}{\draw (\Xy) -- +(135:\RSu) node{};}
\RSdef{r}{\draw (\Xy) -- +(45:\RSu) node{};}
\RSdef{I}{\draw (\Xy) -- +(90:\RSu) coordinate(\Xy I);\edef\Xy{\Xy I}}
\RSdef{L}{\draw (\Xy) -- +(135:\RSu) coordinate(\Xy L);\edef\Xy{\Xy L}}
\RSdef{R}{\draw (\Xy) -- +(45:\RSu) coordinate(\Xy R);\edef\Xy{\Xy R}}

\def\OU{\RS{i}}
\def\OUw{\RS{w}}

\makeatletter
\newcommand{\ogeneric}[2][0.7]{%
  \vphantom{\oplus}\mathpalette\o@generic{{#1}{#2}}%
}
\newcommand{\o@generic}[2]{\o@@generic#1#2}
\newcommand{\o@@generic}[3]{%
  \begingroup
  \sbox\z@{$\m@th#1\oplus$}%
  \dimen@=\dimexpr\ht\z@+\dp\z@\relax
  \savebox\tw@[\totalheight]{$\m@th#1\bigcirc$}%
  \makebox[\wd\z@]{%
    \ooalign{%
      $#1\vcenter{\hbox{\resizebox{\dimen@}{!}{\usebox\tw@}}}$\cr
      \hidewidth
      $#1\vcenter{\hbox{\resizebox{#2\dimen@}{!}{$#1\vphantom{\oplus}{#3}$}}}$%
      \hidewidth
      \cr
    }%
  }%
  \endgroup
}
\makeatother

\def\CM{\mathrm{CM}}
\def\eqdef{\coloneq}

\newcommand{\ogeq}{\mathrel{\rlap{\kern0.18em\tikz[x=0.1em,y=0.1em,baseline=0.04em] \draw[line width=0.25pt] (0,0) -- (23:3.7);}\ogreaterthan}}

\DeclareMathOperator{\id}{id}

\makeatletter
\@namedef{subjclassname@2020}{%
  \textup{2020} Mathematics Subject Classification}
\makeatother

\begin{document}
\baselineskip = 14pt

\title
{Quasi-Gaussianity of the 2D stochastic\\ Navier--Stokes equations}

\author
{James Coe$^1$, Martin Hairer$^2$\orcidlink{0000-0002-2141-6561}, Leonardo Tolomeo$^1$\orcidlink{0000-0002-1882-2294}}

\institute{The University of Edinburgh, UK\\ \email{j.coe@ed.ac.uk, l.tolomeo@ed.ac.uk}
\and
EPFL, Switzerland and Imperial, UK\\ \email{martin.hairer@epfl.ch, m.hairer@imperial.ac.uk}}
%
%
%
%
%
%
%

\maketitle

\begin{abstract}
We study the qualitative properties of solutions to the 2D stochastic 
Navier--Stokes equations with forcing that is white in time and 
coloured in space. Our main result shows that the unique invariant measure 
of this system is equivalent to that of the corresponding Ornstein--Uhlenbeck process. 

Our method relies on a generalization of the ``time-shifted Girsanov method'' 
of \cite{Mattingly2005-aj,MRS22} to compare the laws of time marginals for dissipative SPDEs. 
This generalisation allows to not only compare solutions to a nonlinear equation to those
of the corresponding linear equation, but also to directly compare two nonlinear equations.
We use this to establish equivalence of the Navier--Stokes system to a ``twisted'' nonlinear system that leaves the Gaussian measure invariant. We further apply this method to establish similar equivalence statements for a family of hypoviscous Navier--Stokes equations.
\end{abstract}

\thispagestyle{empty}

\tableofcontents

\section{Introduction}

Consider the stochastic Navier--Stokes equations on the torus $\T^2$ driven by additive noise
\begin{equ}[e:SNS]
        \partial_t u+(u\cdot\nabla)u=\Delta u-\nabla p+ \sqrt2\boldsymbol{\xi}_\alpha, \qquad
        \nabla\cdot u=0,
\end{equ}
for velocity $u:\R\times\T^2\rightarrow\R^2$ and noise 
$\boldsymbol{\xi}_\alpha=|\nabla|^{-\alpha}\boldsymbol{\xi}_0$, 
where $\alpha>0$ and $\boldsymbol{\xi}_0$ is a divergence-free vector-valued spacetime white noise. 
We will assume throughout that both the noise $\boldsymbol{\xi}_0$ and the initial condition $u_0$ are
of vanishing mean, so that the solution to \eqref{e:SNS} has vanishing mean for all times.
In particular, expressions like $|\nabla|^{-\alpha}u_t$ are well defined and the heat semigroup
converges to $0$ at an exponential rate.

The model \eqref{e:SNS} (possibly with additional Ekman damping which has no incidence on the problems
discussed in this article) is a popular model for 2D turbulence and has been widely studied in the 
mathematical literature. Of particular interest is its long-time behaviour, including the study of its
(usually unique) invariant measure $\rho$, see \cite{ErgodNS,DPD,Ee-Mattingly-Sinai,MR1868991,MR1785459,MR2259251,MR2478676} 
for some milestones in this direction. 
While conditions for the existence and uniqueness of such invariant measures are by now very well understood
(but see \cite{MR4244269} for an open problem in this direction), it is a hard problem
to understand their qualitative features beyond basic regularity and integrability properties. This
is due in large part to the fact that the Markov process generated by \eqref{e:SNS} is not reversible.

One natural question one may ask is whether the invariant measure of \eqref{e:SNS} ``looks'' 
approximately Gaussian, in other words whether it is equivalent to the Gaussian measure $\mu_\alpha$ that is 
invariant for the linear equation 
\begin{equ}[e:SNSlin]
        \partial_t u=\Delta u+\sqrt2\boldsymbol{\xi}_\alpha\;.
\end{equ}
In the special case $\alpha = 0$ (with the driving noise being \textit{exactly} space-time white noise
and not merely a Gaussian noise with a covariance structure yielding similar qualitative properties),
the answer is not only affirmative, but the two invariant measures actually happen to
coincide \cite{DPD}. This is a very special situation and there is no reason to believe that these
measures are the same in general, but the main idea of this article is to manufacture a similar
situation when $\alpha \neq 0$ and to show that the solutions to the stochastic Navier--Stokes equation
are close enough to that manufactured situation.

This general question was first studied in \cite{Mattingly2005-aj} where the authors showed that 
the invariant measures of \eqref{e:SNS} and \eqref{e:SNSlin} are equivalent, provided that the
Laplacian $\Delta$ is replaced by a ``hyperviscous'' term of the form
$-(-\Delta)^{1+\delta}$ for some $\delta > 0$ (and provided that $\alpha$ is large enough). 
The ``time-shifted Girsanov method'' introduced in that work was further extended in 
\cite{Suidan2,Girsanov2,MRS22} to more general settings, covering in particular also much lower regularity regimes.
A very recent work \cite{Singular} shows that this method
is in general sharp, namely there exists a large class of stochastic PDEs such that their
invariant measures are equivalent to the corresponding Gaussian measure if and only if the natural analogue
of the condition given in \cite{Mattingly2005-aj} holds. This would in principle suggest that $\rho$ and
$\mu_\alpha$ are mutually singular in the case of \eqref{e:SNS}. The argument given in
\cite{Singular} however relies on exhibiting an explicit quantity which is shown to converge under the Gaussian
measure and to diverge under the invariant measure of the nonlinear equation.
In the specific case of the Navier--Stokes nonlinearity, it turns out that additional cancellations
cause this quantity to vanish, so that the question remained open.

The main result of this article shows that the structure of the Navier--Stokes nonlinearity is such that
equivalence to the Gaussian measure holds beyond the natural threshold identified in \cite{Mattingly2005-aj}.
More precisely, we have the following result.

\begin{theorem} \label{theo:main}
For every $\alpha > 0$, \eqref{e:SNS} admits a unique invariant measure $\rho$, which is equivalent to $\mu_\alpha$.
\end{theorem}

This leads to the question of how far this can actually be pushed. 
In order to investigate this, we consider the \textit{hypo}viscous
stochastic Navier--Stokes equations, namely
\begin{equ}[e:SNSgen]
        \partial_t u+(u\cdot\nabla)u=-|\nabla|^{2\gamma} u-\nabla p+\sqrt2\boldsymbol{\xi}_{\alpha + 1-\gamma}, \qquad
        \nabla\cdot u=0,
\end{equ}
for $\gamma \le 1$. Here the regularity index of the noise is chosen in such a way that solutions still
take values in $\CC^{\alpha - \kappa}$ for every $\kappa > 0$. It furthermore has the 
property that the invariant measure for the corresponding linear equation coincides with the
invariant measure $\mu_\alpha$ for \eqref{e:SNSlin}.

As before, \eqref{e:SNSgen} is locally well-posed in $\CC^{\alpha-\kappa}$ for every $\kappa>0$ as soon 
as $\alpha > 0$, provided
that $\gamma > \f12$ so that the dissipative term dominates the transport term. We will therefore always
make that assumption in the sequel. It is however more difficult to obtain global a priori
bounds when $\gamma < 1$ due to the fact that the equation is then energy-supercritical. As a consequence,
we will restrict ourselves to large enough values of $\alpha$ so that one can use 
enstrophy bounds to get global well-posedness.

Our main result then shows that the 
transition probabilities of \eqref{e:SNSgen} are also equivalent to the invariant measure 
of the corresponding Ornstein--Uhlenbeck process, at least when $\gamma > 2/3$ and $\alpha$ is large enough.

\begin{theorem} \label{theo:mainGeneral}
Let $\gamma \in (2/3,1]$ and $\alpha > 2-\gamma$. Then \eqref{e:SNSgen} admits global solutions 
and a
unique invariant measure $\rho$ in $\CC^{\alpha - \kappa}$ for every $\kappa \in (0,\alpha)$. 
Furthermore, $\rho$ is equivalent to $\mu_\alpha$.
\end{theorem}

\begin{remark}
One would expect to be able to adapt the proofs given in \cite{TommasoNS,WenhaoNS} to obtain 
global well-posedness and the existence of a unique invariant measure for all $\alpha > 1$ in 
the regime $\gamma \in (\f23,1)$, but this is not the focus of the present work.
\end{remark}

\begin{remark}
One reason for the condition $\gamma > 2/3$ is that we are no longer able to prove that solutions 
to \eqref{e:SNSgen} are strong Feller when it fails. This is because when $\gamma \le 2/3$ the linearised 
equation fails to be sufficiently regular for our application of the Bismut--Elworthy--Li formula.
The same lack of regularity of the Jacobian however also causes a crucial step in our proof to fail, 
see the end of the next section.
\end{remark}

\begin{remark}
    We shall see that the constraint $\alpha>2-\gamma$ only arises to ensure \eqref{e:SNSgen} has global solutions (and an invariant measure), and we believe our abstract equivalence result Proposition \ref{prop:equiv} may be applied to a lower-regularity regime. Namely for $\gamma\in(2/3,1)$ and $\alpha>2-2\gamma$, for any $t>0$ we have
    \begin{equation*}
        \Law\bigl(u(t);T^*>t\bigr)\ll \mu_\alpha\;,
    \end{equation*}
where $\Law\bigl(u(t);T^*>t\bigr)$ is the law of the solution at time $t$ restricted to the event 
where solutions have not blown up yet (in the space $\CC^{\alpha-\kappa}$).
\end{remark}

The remainder of the article consists of three parts. In Section~\ref{sec:Idea}, we recall a few existing
results that we will make use of and we explain the idea of the proof of our main result.
Section~\ref{sec:bounds} is devoted to some a priori bounds that are useful for the main proof.
In particular, these bounds guarantee the existence of a unique invariant measure in
the regimes of interest to us. Finally, we provide the proof of Theorems~\ref{theo:main} and~\ref{theo:mainGeneral} 
in Section~\ref{sec:proof}.

\subsection{Notations}

We let $\T^2=(\R/2\pi\Z)^2$ with normalised integral $\int_{\T^2}dx=1$. For $f\in\S'(\T^2;\R^2)$, we denote the Fourier transform
\begin{equ}
    \ft{f}(n)=\F[f](n)=\int_{\T^2}f(x)e^{in\cdot x}\,dx\;,
\end{equ} for $n\in\Z^2$. Throughout we shall let $L^p,H^\beta,\CC^\beta$ denote the closure of smooth, divergence-free and mean-free vector fields under the norms $L^p(\T^2;\R^2),\,H^\beta(\T^2;\R^2),\,\CC^\beta(\T^2;\R^2)\,$ respectively, where we recall the H\"older-Besov norm \begin{equ}
    \|f\|_{\CC^\beta}=\sup_{N \textit{ dyadic integer}}N^{\beta}\|P_N f\|_{L^\infty}\;,
\end{equ} for $(P_N)_{N \text{ dyadic}}$ a family of Littlewood--Paley projectors on the torus. We let $A\lesssim B$ mean that there is a uniform constant $C>0$ such that $A\leq C B$, and $A\ll B$ means there is a sufficiently small $0<C<1$ such that $A\leq CB$. We let $P$ denote the Leray projector on $\S'(\T^2;\R^2)$, and we recall for divergence-free $f,g$
\begin{equation*}
    P\div \bigl(f \otimes g\bigr)=P\bigl((f\cdot\nabla)g\bigr)\;,
\end{equation*} and the estimate for any $s>0$
\begin{equ}
    \|P\div\bigl(f \otimes g \bigr)\|_{\CC^{s-1}}\lesssim \|f\|_{\CC^s}\|g\|_{\CC^s}\;.
\end{equ} We shall use the notation $\|\cdot\|\eqdef\|\cdot\|_{L^2(\T^2)}$, and also abbreviate path-space norms $L^p([0,T];X)$ as $L^p_TX$ or $L^p_TX_x$, for some Banach space $X$ of functions on $\T^2$. For $\gamma>0$ we let $P^{(\gamma)}_t$ denote the semigroup generated by $-|\nabla|^{2\gamma}$ on the spaces $\CC^\beta$, with $\gamma=1$ corresponding to the heat semigroup (restricted to divergence and mean-free functions). We note that $P^{(\gamma)}$ forms a strongly continuous semigroup on $\CC^\beta$, and similarly to the kernel estimates of \cite{Wu01,ezz08}, we see that for any $a\in\R$, $b\geq 0$
\begin{equ}\label{e:semigroupbounds}
    \|P^{(\gamma)}_{t}f\|_{\CC^{a+b}}\lesssim t^{-\frac{b}{2\gamma}}\|f\|_{\CC^{a}}\;,
\end{equ}
for all $t>0$ and $f\in\CC^{a}$.

For Banach spaces $X,Y$, we say $F:X\rightarrow Y$ is \textit{uniformly smooth} (from $X$ to $Y$) if it is infinitely differentiable (in the Fr\'echet sense) with all derivatives bounded on bounded sets. Note that the derivatives of a uniformly smooth function are Lipschitz on bounded sets (which may not hold if the
function is merely Fréchet smooth). We say that a collection $(F_s)_{s\in[0,1]}$ of maps $F_s:X\rightarrow Y$ is a \textit{uniformly smooth path} if $\sup_{\|x\| \le R} \sup_{s \in [0,1]} \|D^k F_s(x)\| < \infty$
for every $R > 0$ and $k \ge 0$.

\subsection*{Acknowledgements}

{\small
The work of MH was partially supported by the Simons Collaboration 
``probabilistic paths to QFT''.
He would also like to thank Jonathan Mattingly for many long discussions on this problem.
}

\section{Idea of proof}
\label{sec:Idea}

Before we start sketching the idea of our proof, we collect in one statement the well-posedness
results we require for our questions to even be well posed.

\begin{proposition}\label{prop:apriori}
Let $\gamma \in (1/2,1]$ and let $\alpha > 0$. Then, \eqref{e:SNSgen} admits unique local mild solutions
in $\CC^{\alpha - \kappa}$ for every $\kappa \in (0,\alpha)$.
If either $\gamma = 1$ and $\alpha > 0$ or $\gamma \in (1/2,1)$ and $\alpha > 2-\gamma$, then these solutions
are global in time and \eqref{e:SNSgen} admits an invariant probability measure on $\CC^{\alpha - \kappa}$.
If furthermore $\gamma > 2/3$, then the Markov semigroup generated by \eqref{e:SNSgen} has the strong Feller
property and is topologically irreducible. In particular, the invariant measure $\rho$ is unique
and all transition probabilities are equivalent to $\rho$.
\end{proposition}

\begin{proof}
The local solution theory for \eqref{e:SNSgen} in $\CC^{\alpha-\kappa}$ for 
every $\kappa > 0$ follows from standard techniques, see
for example \cite{Da_Prato_Zabczyk_2014} or \cite{SPDENotes}.
The remainder of the statement is a combination of Lemmas~\ref{lem:goodBounds} and~\ref{lem:apriori}, 
as well as the proof of Proposition~\ref{prop:SF} below.
These are by now quite standard so we will not dwell on them, but we provide 
sketches of proofs in Section~\ref{sec:bounds}.
\end{proof}

\begin{remark}
One can in fact obtain a local solution theory for \eqref{e:SNS} in $\CC^{\alpha-\kappa}$ in a suitable
renormalised sense for every $\alpha > -1$ as a consequence of the general theory developed in
\cite{RegStruc,AlgRen,GenRen,Rhys}. It is however not known in general
whether these solutions are global (and a fortiori whether they admit an invariant measure) 
when $\alpha < 0$. See however \cite{DPD,WenhaoNS} for the borderline case $\alpha = 0$.
\end{remark}

\begin{remark}
It was shown in \cite{Franco} that \eqref{e:SNS} admits global solutions and an invariant
measure when $\alpha > 1/2$. In \cite{Benedetta,ErgodNS}, it was furthermore shown
that the Markov semigroup generated by \eqref{e:SNS} is strong Feller
and topologically irreducible when $\alpha > 1$, and therefore that its invariant measure is unique.
We show that current techniques allow to extend these results to the range $\alpha > 0$
and we will cover the hypoviscous case $\gamma < 1$ provided that $\alpha > 2-\gamma$.
The reason why we need one additional derivative in this case is that the $L^2$ norm becomes
supercritical when $\gamma < 1$ and we need to rely on enstrophy bounds instead.
\end{remark}

%

Let us first rewrite \eqref{e:SNSgen} in vorticity form as
\begin{equ}[e:vort]
        \partial_t w=-|\nabla|^{2\gamma} w- (\CK w\cdot\nabla)w + \sqrt2\xi_{\alpha -\gamma}\;,
\end{equ}
where $\xi_{\alpha -\gamma} = \nabla \wedge \boldsymbol{\xi}_{\alpha + 1 -\gamma}$ and $\CK w$ is the unique divergence-free and mean-free vector field
such that $\nabla \wedge \CK w = w$.
Consider then instead the equation
\begin{equ}[e:equv]
        \partial_t v=-|\nabla|^{2\gamma} v- |\nabla|^{-\alpha}(\CK v\cdot\nabla)|\nabla|^{\alpha}v + \sqrt2\xi_{\alpha -\gamma}\;.
\end{equ}
Setting $\hat v = |\nabla|^\alpha v$ we then see that, at least formally, $\hat v$ solves
\begin{equ}[e:hatv]
        \partial_t \hat v=-|\nabla|^{2\gamma} \hat v- (\hat \CK \hat v\cdot\nabla)\hat v + \sqrt2\xi_{-\gamma}\;,
\end{equ}
where $\hat \CK = \CK |\nabla|^{-\alpha}$. This is designed in such a way that the 
corresponding Ornstein--Uhlenbeck process
\begin{equ}
\partial_t \psi =-|\nabla|^{2\gamma} \psi + \sqrt2\xi_{-\gamma}\;,
\end{equ}
admits the white noise measure as its unique invariant measure.
Since $\hat \CK \hat v$ is divergence free, one has
$\scal{\hat v,(\hat \CK \hat v\cdot\nabla)\hat v} = 0$, which strongly suggest that 
the white noise measure is also invariant for \eqref{e:hatv} in the same way as in \cite{DPD}.

This immediately implies that the Markov semigroup generated by \eqref{e:equv} also admits 
the same invariant measure as 
the corresponding Ornstein--Uhlenbeck process. When $\gamma > 2/3$, we can show
that the corresponding Markov semigroup is 
topologically irreducible and strong Feller, so that its transition probabilities
are all equivalent to its invariant measure \cite{DPZ2}. The claim then follows
if we can show that the transition probabilities for \eqref{e:vort} are equivalent to those
for \eqref{e:equv}. Similarly to the methodology used in \cite{Mattingly2005-aj,MRS22,Singular}, 
the idea is to construct a process $Z$ such that the solution to 
\begin{equ}[e:vortshift]
        \partial_t w=-|\nabla|^{2\gamma} w- (\CK w\cdot\nabla)w + Z + \sqrt2\xi_{\alpha -\gamma}\;,
\end{equ}
coincides with that of \eqref{e:equv} at some fixed final time $T$ and then make use
of Girsanov's theorem. For this, we need to make sure that $Z$ is adapted to the
filtration generated by $\xi$ and belongs to its Cameron--Martin space, namely
that one has $Z \in L^2([0,T],H^{\alpha-\gamma})$ almost surely.

A ``naive'' choice for $Z$ would be to simply take
\begin{equ}
Z = (\CK w\cdot\nabla)w - |\nabla|^{-\alpha}(\CK w\cdot\nabla)|\nabla|^\alpha w\;.
\end{equ}
The ``fractional Leibniz rule'' suggests that, at least when $\alpha \le 1$, this 
expression has the same regularity properties as  
$|\nabla|^{-\alpha}(|\nabla|^\alpha \CK w\cdot\nabla) w$ while when $\alpha > 1$
we expect it to behave like $|\nabla|^{-\alpha}(|\nabla| \CK w\cdot\nabla) |\nabla|^{\alpha-1} w$. 
Since we expect $w$ to
have regularity (almost) $\alpha-1$, one could expect $Z$ to have 
 regularity almost $(\alpha-1) \wedge (2\alpha-2)$ (but no better), which is however not good enough to
 be able to apply Girsanov since, when $\gamma < 1$, this does not belong to the
 Cameron--Martin space $L^2_T H^{\alpha - \gamma}$.

The idea in \cite{Mattingly2005-aj} is to observe that, at time $T$, the solution to
\begin{equ}[e:X]
\d_t X = -|\nabla|^{2\gamma} X + Z(t) + \sqrt2\xi_{\alpha-\gamma}\;,\quad X(0) = X_0\;,
\end{equ}
coincides with that to
\begin{equ}
\d_t Y = -|\nabla|^{2\gamma} Y + 2P_{T-t}^{(\gamma)} Z(2t-T)\one_{2t > T} + \sqrt2\xi_{\alpha-\gamma}\;,\quad Y(0) = X_0\;.
\end{equ}
This is because, writing $\OUw$ for the solution to \eqref{e:X} with $Z=0$ and 
$P_{t}^{(\gamma)}$ for the semigroup generated by $-|\nabla|^{2\gamma}$, one has
\begin{equs}
X(T) &= \OUw + 
\int_0^T P_{T-s}^{(\gamma)} Z(s)\,ds
= \OUw + 2\int_{T/2}^T P_{2(T-r)}^{(\gamma)} Z(2r-T)\,dr\\
&= \OUw + 2\int_{T/2}^T P_{T-r}^{(\gamma)} P_{T-r}^{(\gamma)} Z(2r-T)\,dr = Y(T)\;.\label{e:SGTrick}
\end{equs}
Note now that if $Z \in \CC^{\alpha-1}$, then 
$\|P_t^{(\gamma)} Z\|_{\CC^{\alpha-\gamma}} \lesssim t^{\f{\gamma-1}{2\gamma}}$ which is
square integrable as soon as $\gamma > \f12$. In the low regularity regime $\alpha \le 1$,
we expect to have $Z \in \CC^{2(\alpha-1)}$ and a similar calculation yields the condition
$\gamma > 1-\f\alpha2$.

The problem with this argument is that \eqref{e:vortshift} is not actually of the form
\eqref{e:X} due to the presence of the nonlinearity. This difficulty can be overcome
in a way reminiscent of \cite{SFGeneral} by applying the above argument to the linearisation
of \eqref{e:vort} and then integrating up the ``infinitesimal Girsanov shifts''
produced in this way. The reason why we need to impose the more stringent
restriction $\gamma > \f23$ is that the linearisation $J$ of \eqref{e:vortshift} 
solves
\begin{equ}
\d_t J_{s,t} = -|\nabla|^{2\gamma} J_{s,t} - (\CK w\cdot\nabla)J_{s,t} - (\CK J_{s,t}\cdot\nabla)w\;,
\quad J_{s,s} = \id\;.
\end{equ}
As a consequence of the presence of the last term, we do not expect the regularity of $J$ to ever
go beyond  $2\gamma-1$ derivatives more than that of $w$.
Since otherwise the short-time regularisation properties of $J$ are analogous to those of the 
semigroup  $P_{t}^{(\gamma)}$, we expect the above argument to work as long as
$(2\gamma -1) + (\alpha-1) > \alpha - \gamma$, leading to the condition $\gamma > \f23$.

\remark{It is interesting to consider noises with more general covariance operators. Assuming Gaussian noise with Cameron--Martin space $L^2_TH$, letting $Q$ be the (spatial) covariance operator, then the above argument requires that the commutator
\begin{equ}
    C_Q[w]=(\CK w\cdot\nabla)w - Q(\CK w\cdot\nabla)Q^{-1} w
\end{equ} maps solutions into $|\nabla|^{\gamma}H$. Just assuming $Q\approx |\nabla|^{\gamma-1-\alpha}$ is not sufficient. Consider
\begin{equ}
    \ft{Qf}(n)=\s_n |n|^{\gamma-1-\alpha}\ft{f}(n),
\end{equ} for weights $\s_n=2+(-1)^{|n|}$, then we have $|Q|\approx |\nabla|^{\gamma-1-\alpha}$ but the commutator estimate fails. In this case, the diverging quantities from \cite{Singular} indicate singularity. So equivalence relies on more structure of the covariance operator than just Sobolev mapping properties, namely the smoothness of the symbol.}

\section[Properties of solutions]{Properties of the solution to stochastic Navier--Stokes equations}
\label{sec:bounds}

In this section, we obtain some a priori bounds on \eqref{e:SNSgen} and its
linearisation. These bounds are obtained in a way that is very similar to the aforementioned 
bounds for the  ``normal'' stochastic Navier--Stokes equations \eqref{e:SNS}, but the
uniform tightness bound (which yields the existence of an invariant measure) is slightly more
delicate than \cite{Franco}, so we focus on that.

\subsection{Bounds on stochastic convolution}

Similarly to the case of the usual Navier--Stokes equations, one can rewrite \eqref{e:SNSgen} as
\begin{equ}
        \partial_t u=-|\nabla|^{2\gamma} u-P\div \bigl(u \otimes u\bigr)+\sqrt2\boldsymbol{\xi}_{\alpha + 1-\gamma},
\end{equ}
where $P$ denotes the Leray projection. We exploit the following trick previously used in 
\cite{Franco}. For $K>0$ (which we think of as large), we write $\OU$ for the stationary solution to the linear equation
with an additional damping of size $K$, namely
\begin{equ}[e:defOUK]
\OU(t) = \sqrt2\int_{-\infty}^t e^{-K(t-s)} \CP^{(\gamma)}_{t-s}\,\boldsymbol{\xi}_{\alpha + 1-\gamma}(s,\cdot)\,ds\;.
\end{equ}
The idea will be to make $K$ large enough so that $\OU$ is small in suitable norms. One has the following
bound.

\begin{lemma}\label{lem:OU}
For $\kappa > 0$ small enough, there exists a constant $C>0$ such that
\begin{equ}
\E \sup_{t \in (0,1]} K^{\kappa} \|\OU(t)\|_{\CC^{\alpha - 2\kappa}} \le C\;,
\end{equ}
holds uniformly over all $K\ge 1$.
\end{lemma}

\begin{proof}
Writing $\zeta_k(t)$ for the $k$th Fourier coefficient of $\OU(t)$, we note that
these are jointly Gaussian and independent for different values of $k$. Furthermore,
one has
\begin{equ}
\E \zeta_k(t) \zeta_k(s) = \frac{|k|^{2\gamma-2\alpha-2}}{K+|k|^{2\gamma}} \exp \bigl(-(K+|k|^{2\gamma})|t-s|\bigr)\;.
\end{equ}
In particular, it follows that
\begin{equs}
\E |\OU(t,x)-\OU(s,x)|^2
&= \sum_k \frac{2|k|^{2\gamma-2\alpha-2}}{K+|k|^{2\gamma}} \Bigl(1-\exp \bigl(-(K+|k|^{2\gamma})|t-s|\bigr)\Bigr)\\
&\le \sum_k \frac{2|k|^{2\gamma-2\alpha-2}}{K+|k|^{2\gamma}} \bigl(1 \wedge (K+|k|^{2\gamma})|t-s|\bigr)\\
&\le \sum_k \frac{2|k|^{2\gamma-2\alpha-2}}{(K+|k|^{2\gamma})^{1-\kappa}} |t-s|^\kappa\\
&\le 2K^{-2\kappa} |t-s|^\kappa \sum_k |k|^{3\kappa -2\alpha-2} \le CK^{-2\kappa} |t-s|^\kappa\;.
\end{equs}
For $\alpha \le 1$, one similarly (see for example \cite[Cor.~4.24]{SPDENotes}) has
\begin{equs}
\E |\OU(t,x)-\OU(t,y)|^2 &\le 2|x-y|^{2\alpha-3\kappa} \sum_k \frac{|k|^{2\gamma-3\kappa-2}}{K+|k|^{2\gamma}}\\
&\le 2K^{-2\kappa} |x-y|^{2\alpha-3\kappa} \sum_k |k|^{-\kappa-2}\;,
\end{equs}
and the claim follows at once from Kolmogorov's criterion. For larger values of $\alpha$, it suffices to apply
the same argument to the derivatives of $\OU$.
\end{proof}

\begin{remark}
It follows immediately from Fernique's theorem \cite{Fernique,SPDENotes} that one has for every $p\ge 1$ the stronger bound
\begin{equ}
\E \sup_{t \in (0,1]} K^{\kappa p} \|\OU(t)\|^p_{\CC^{\alpha - 2\kappa}} \le C\;,
\end{equ}
where $C$ depends of course on $p$, but not on $K$.
\end{remark}

We then write $u = \OU + v$, so that $v$ solves the random PDE
\begin{equ}
\d_t v = -\abs\nabla^{2\gamma}v - P \div ((\OU + v)\otimes (\OU + v)) + K \OU\;.
\end{equ}
We also write $\|v\|$ for the $L^2$-norm of $v$ and $\|v\|_1$ for its homogeneous $H^1$ norm.
Since we restrict ourselves to divergence-free vector fields, we have $\|v\|_1 = \|\nabla \wedge v\|$.
Write now $\OUw = \nabla \wedge \OU$ and $w = \nabla \wedge v$, so that
\begin{equ}[e:equw]
\d_t w = -\abs\nabla^{2\gamma}w - ((\OU + v) \cdot \nabla) (\OUw + w) + K \OUw\;.
\end{equ}
\begin{lemma}\label{lem:apriori0}
Let $\gamma \in (\f12,1]$ and $\alpha > 2-\gamma$. Then, there exists $p>2$ such that one has the bound
\begin{equ}
\d_t \|w\|^2 \le -\|\abs\nabla^\gamma w\|^2 + C\|w\|^2 \|\OUw\|_{\CC^{1-\gamma}}^{2} + C_K(1+\|\OUw\|_{\CC^{1-\gamma}}^{4} )\;.
\end{equ}
\end{lemma}

\begin{proof}
It follows from \eqref{e:equw} and the fact that $u = \OU + v$ is divergence free that
\begin{equ}
\d_t \|w\|^2 = -2 \|\abs\nabla^\gamma w\|^2 + 2\scal{\OUw,u \nabla w}
+ 2K \scal{w,\OUw}\;.
\end{equ}
Since $\nabla / \abs\nabla$ is a bounded operator from $L^2(\T^2, \R)$ into $L^2(\T^2, \R^2)$,
we have
\begin{equ}
|\scal{\OUw,u \nabla w}| \lesssim \|\abs\nabla^\gamma w\| \|\abs\nabla^{1-\gamma} (u\,\OUw)\|\;.
\end{equ}
The fractional Leibniz rule \cite[Thm~A.8]{KPV} (see \cite[Prop~1]{Benyi24} for the torus case) in turn shows that
\begin{equs}
\|\abs\nabla^{1-\gamma} (u\,\OUw)\|
&\lesssim \|\abs\nabla^{1-\gamma} u\| \|\OUw\|_{L^\infty}
+ \| u\|_{L^4} \|\abs\nabla^{1-\gamma} \OUw\|_{L^4} \\
&\lesssim \bigl(\|w\| + \|\OUw\|\bigr) \|\OUw\|_{\CC^{1-\gamma}}\;.
\end{equs} 
Since furthermore
\begin{equ}
|\scal{w,\OUw}| \lesssim \|w\| \|\OUw\|_{\CC^{1-\gamma}}\;,
\end{equ}
the claim follows at once.
\end{proof}

It turns out that this bound is strong enough to imply that the process $u$ has an invariant
measure with good moment bounds:

\begin{lemma}\label{lem:goodBounds}
Let $\gamma \in (\f12,1]$, $\alpha > 2-\gamma$, and $\kappa > 0$. Then, \eqref{e:SNSgen}
admits unique global mild solutions taking values in $\CC^{\alpha-\kappa}$, as well
as an invariant measure $\rho$ such that
\begin{equ}[e:goodIM]
\int \|u\|_{\CC^{\alpha-\kappa}}^p\,\rho(du) < \infty\;,
\end{equ}
for every $p>0$.
\end{lemma}

\begin{proof}
A standard Picard iteration argument (see for example \cite[Thm~7.8]{SPDENotes}) 
and the fact that $\OU$ takes values in $\CC^{\alpha-\kappa}$ imply that, for $\delta > 0$
small enough, \eqref{e:SNSgen} admits unique local mild solutions with values in $\CC^{\alpha-\kappa}$
for every initial condition $u_0 \in \CC^{-\delta}$.

Assume now that we know that these solutions are global and admit an invariant measure under which
the $H^1$ norm of $u$ (and therefore also its $\CC^{-\delta}$ norm) has finite moments of all orders. 
Writing again $u$ for the stationary solution to \eqref{e:SNSgen}, we can then use repeatedly the 
fact that, for every norm $\$\bigcdot\$$, one has
\begin{equ}
\$u_1\$ \le \$\CP^{(\gamma)}_1 u_0\$ + \int_0^1 \$\CP^{(\gamma)}_{1-s} P\div (u_s \otimes u_s)\$\,ds
+ \$\OU\$\;,
\end{equ}
to show that \eqref{e:goodIM} does indeed hold.

In order to obtain the required moment bounds on the $H^1$ norm, we make use of Lemma~\ref{lem:apriori0}.
In view of that result, it suffices to show that, for every $C>0$ there exists
a value of $K$ such that 
\begin{equ}[e:wanted]
\E \sup_{t>0} e^{-t}\exp \Bigl(C\int_0^t \|\OUw(s)\|_{\CC^{1-\gamma}}^{2}\,ds \Bigr) < \infty\;.
\end{equ}
(Recall that $\OU$ and $\OUw$ are defined by \eqref{e:defOUK} and therefore depend on $K$.)
In view of Lemma~\ref{lem:OU} and Fernique's theorem, we conclude that for any 
$C>0$ and any $\delta > 0$ there exists $K$ such that 
\begin{equ}[e:boundOne]
\E \exp \Bigl(C \sup_{s \in [0,1]} \|\OU(s)\|_{\CC^{2-\gamma}}^{2}\Bigr) \le 1+\delta \;.
\end{equ}
A similar bound also holds for $\hat \OU$, defined like $\OU$ but with the time 
integral starting at $0$ rather than $-\infty$. If we write $\hat \OU_i$ for an i.i.d.\ sequence
of copies of $\hat \OU$ indexed by $i \in \N$ with the convention that $\hat \OU_{-1}$ is a copy
of $\OU$ (independent of the other $\hat \OU_i$'s) and we define a process $t \mapsto Z_t$
for $t \ge 0$ by
\begin{equ}
Z_{t} = \hat \OU_i(t-i) + \sum_{k = 0}^{i}e^{-K(t-k)} \CP^{(\gamma)}_{t-k} \hat \OU_{k-1}(1)\;,\qquad t \in [i,i+1)\;,
\end{equ}
then one can check that $Z$ and $\OU \restr \R_+$ have the same law. As a consequence, we 
have the stochastic domination
\begin{equ}
\int_0^k \|\OUw(s)\|_{\CC^{1-\gamma}}^{2}\,ds
\lesssim \sum_{i=0}^{k} \sup_{s \in [0,1]} \|\hat \OU_i(s)\|_{\CC^{2-\gamma}}^{2}\;.
\end{equ}
Combining this with \eqref{e:boundOne} and exploiting the independence of the 
$\hat \OU_i$'s, \eqref{e:wanted} follows at once.
\end{proof}

We also have

\begin{lemma}\label{lem:apriori}
Let $\gamma = 1$, $\alpha > 0$ and $\kappa \in (0,\alpha)$. Then, for every $u_0 \in \CC^{\alpha-\kappa}$, \eqref{e:SNS} 
admits unique global mild solutions in $\CC^{\alpha - \kappa}$. Furthermore, the
resulting Markov semigroup on  $\CC^{\alpha - \kappa}$ admits an invariant measure.
\end{lemma}

\begin{proof}
Local well-posedess in $\CC^{\alpha-\kappa}$ and $L^2$ are classical, see \cite{SPDENotes}
and \cite{Franco} respectively. Global well-posedness in $L^2$ and uniform moment bounds were 
obtained in \cite{WenhaoNS}. To obtain global well-posedness in $\CC^{\alpha - \kappa}$, one 
first proceeds exactly as in \cite[p.~415]{Franco} (Step~4) to show that solutions 
are continuous with values in $H^{\alpha-\kappa}$. Once this is established, one uses the fact that
the $H^{\alpha-\kappa}$ norm is subcritical to conclude with a standard bootstrapping argument.

It remains to show that the equation admits an invariant measure on $\CC^{\alpha - \kappa}$.
By the moment bounds of \cite{WenhaoNS}, we know that there exists an invariant measure $\rho$
on $L^2$ with moments of all orders. Let now $u$ be a stationary solution to \eqref{e:SNS}
and write $u = v + \OU$ as before. Testing $v$ against itself and using the energy conservation
of the nonlinearity, we obtain
\begin{equs}
\d_t \|v\|^2 &= - 2 \|\nabla v\|^2 + 2\scal{v, \div (v\otimes \OU + \OU \otimes v + \OU \otimes \OU)}\\
&\le  - 2 \|\nabla v\|^2 + 2 \|\nabla v\|\,\bigl(2\|v\|\|\OU\|_{L^\infty} + \|\OU\|_{L^\infty}^2\bigr)\\
&\le  - \|\nabla v\|^2 + \bigl(2\|v\|\|\OU\|_{L^\infty} + \|\OU\|_{L^\infty}^2\bigr)^2\\
& \le  - \|\nabla v\|^2 + \|u\|^4 + C \|\OU\|_{L^\infty}^4\;,
\end{equs} 
for some fixed constant $C$. Integrating this between times $0$ and $1$ (say), we have
\begin{equs}
\int_0^1 \|u(s)\|_{\CC^{-\kappa}}^2\,ds
&\lesssim 
\int_0^1 \bigl(\|\OU(s)\|_{\CC^{-\kappa}}^2 + \|\nabla v(s)\|^2\bigr)\,ds\\
&\le \|v(0)\|^2 + \int_0^1 \bigl(\|\OU(s)\|_{\CC^{-\kappa}}^2 + \|u(s)\|^4 + C \|\OU(s)\|_{L^\infty}^4\bigr)\,ds \;.
\end{equs}
Integrating against $\rho$ and using the stationarity of $u$ and $\OU$, we find that 
$\|u\|_{\CC^{-\kappa}}$ has finite second moment under $\rho$. Since the $\CC^{-\kappa}$ norm
is subcritical, a simple bootstrapping
argument shows that \eqref{e:SNS} admits local solutions in $\CC^{\alpha-\kappa}$ with
explicit control in terms of $\|u(0)\|_{\CC^{-\kappa}}$ and $\sup_{t\in [0,1]}\|\OU(s)\|_{\CC^{\alpha-\kappa}}$. It follows immediately that the $\CC^{\alpha-\kappa}$ norm of $u$ is almost 
surely finite under $\rho$ as claimed.
\end{proof}

\begin{remark}\label{rem:SM}
The driving noise $\xi := \xi_\alpha$ satisfies the ``strong Markov'' property in the following way.
Given $\tau \in \R$, write $S_\tau$ for the shift operator acting on test functions by
\begin{equ}
\bigl(S_\tau \phi\bigr)(t,x) = \phi(t-\tau,x)\;.
\end{equ}
Write furthermore $\CF$ for the filtration given by 
\begin{equ}
\CF_t = \sigma\{\xi_\alpha(\phi)\,:\, \supp \phi \subset (-\infty,t]\times \T^2\}\;,
\end{equ}
and consider a stopping time $\tau$ with respect to this filtration. 
Then, for any finite collection $\{\phi_i\}_{i=1}^M$
of test functions with $\supp \phi_i \subset \R_+\times \T^2$ and for any bounded continuous
function $F\colon \R^M \to \R$, one has
\begin{equ}
\E \bigl(F\bigl(\xi(S_\tau\phi_1),\ldots,\xi(S_\tau\phi_M)\bigr)\,\big|\, \CF_\tau\bigr) = \E F\bigl(\xi(\phi_1),\ldots,\xi(\phi_M)\bigr)\;,
\end{equ}
almost surely.
Since the solution to \eqref{e:SNS} is well-posed, adapted, and only depends on such random variables, 
it follows in a straightforward way that the solutions to \eqref{e:SNS} actually satisfy the strong Markov property.
\end{remark}

\subsection{Linearised equation}

We write the linearised SPDE in abstract form as
\begin{equ}
\d_t J = -|\nabla|^{2\gamma}J + DF(u)J\;,
\end{equ}
where we assume that $u$ takes values in $\CC^{\beta}$ (which is the case for $\beta = \alpha-\kappa$) 
and  that $F$ is uniformly smooth as a map $\CC^{\beta} \to \CC^{\beta-1}$.

\begin{lemma}\label{lem:boundsLinear}
For $T>0$, $\beta\in\R$ and a continuous map $A:[0,T]\rightarrow \B(\CC^\beta ; \CC^{\beta
-1})$, let $(J_{s,t})_{0\leq s\leq t \leq T}$ be the linear propagator of 
\begin{equ}
\d_t J = -|\nabla|^{2\gamma}J + A_tJ\;.
\end{equ} Then for any $\theta\in[0,1)$ and $\delta$ such that
\begin{equ}[e:condexponents]
    \beta-1+2\gamma > \delta+2\gamma\theta> \beta\;,
\end{equ} 
we have
\begin{equ}[e:JstBound]
\|J_{s,t} v\|_{\delta + 2\gamma\theta}\lesssim C_{T,R} |t-s|^{-\theta} \|v\|_{\delta}\;,
\end{equ}
where $R=\sup_{t\in[0,T]} \|A_t\|_{\CC^{\beta} \rightarrow \CC^{\beta-1}}$, and $C_{T,R}\leq \CC^{TR^a}$ for some $C,a>1$.
\end{lemma}

\begin{proof}
Recall that $v_t = J_{s,t} v_s$ solves
\begin{equ}
v_t = P^{(\gamma)}_{t-s} v_s + \int_s^t P^{(\gamma)}_{t-r} A_r\,v_r\,dr\;.
\end{equ}
We now set, for some fixed time $s \ge 0$ and some fixed time interval $T_R>0$
\begin{equ}
\|v\| \eqdef \sup_{|t-s| \le T_R} \bigl(\|v_t\|_{\delta} + (t-s)^\theta\|v_t\|_{\delta + 2\gamma\theta}\bigr)\;.
\end{equ}
We will make use of \eqref{e:semigroupbounds}, that for any $a\in\R$, $b\geq0$ and $\tau>0$
\begin{equ}
    \|P^{(\gamma)}_\tau f\|_{a + b} \lesssim \tau^{-\frac{b}{2\gamma}}\|f\|_{a}\;,
\end{equ}
as well as the fact that for any $a,b>-1$ we have
\begin{equ}
    \int_s^t(t-r)^a(r-s)^b\,dr\propto (t-s)^{1+a+b}\;.
\end{equ} 
As a consequence of \eqref{e:condexponents}, we then obtain
\begin{align*}
    \|v_t\|_{\delta}&\lesssim\|v_s\|_{\delta}+\int_s^t(t-r)^{\frac{\min\left\{0,\beta-1-\delta\right\}}{2\gamma}}\|A_r\,v_r\,\|_{\beta-1}\,dr \\
    &\lesssim\|v_s\|_{\delta}+\int_s^t(t-r)^{\frac{\min\left\{0,\beta-1-\delta\right\}}{2\gamma}}(r-s)^{-\theta}R\|v\|\,dr \\
    &\lesssim\|v_s\|_{\delta}+(t-s)^{(\min\{0,\beta-1-\delta\}+2\gamma-2\gamma\theta)/2\gamma}R\|v\|\;,
\end{align*}
where this exponent is positive. Next we see
\begin{equ}
    (t-s)^{\theta}\|P^{(\gamma)}_{t-s} v_s\|_{\delta +2\gamma\theta} \lesssim \|v_s\|_{\delta}\;.
\end{equ} Lastly we have (for the same conditions on $\delta$)
\begin{align*}
    &(t-s)^\theta\int_s^t \|P^{(\gamma)}_{t-r} A_r\,v_r\,\|_{\delta+2\gamma\theta}\,dr \\&\lesssim (t-s)^\theta\int_s^t (t-r)^{\frac{\beta-1-\delta-2\gamma\theta}{2\gamma}}\|A_r\,v_r\,\|_{\beta -1}\,dr \\
    &\lesssim (t-s)^\theta\int_s^t (t-r)^{\frac{\beta-1-\delta-2\gamma\theta}{2\gamma}}(r-s)^{-\theta}R\|v\|\,dr \\
    &\lesssim(t-s)^{(\beta-1+2\gamma-\delta-2\gamma\theta)/2\gamma}R\|v\|\;,
\end{align*} which is again a positive exponent. And so
\begin{equ}
    \|v\|\lesssim \|v_s\|_\delta + T_R^{(\beta-1+2\gamma-\delta-2\gamma\theta)/2\gamma}R\|v\|\;.
\end{equ}
Since the exponent appearing in this expression is positive, we can choose 
$T_R\propto R^{-a}$ for some $a>0$ sufficiently large so that $\|v\|\lesssim \|v_s\|_\delta$. We can then iterate this to see for $|t-s|\leq T$
\begin{equ}
    (t-s)^{\theta}\|v_t\|_{\delta + 2\gamma\theta}\leq C^{TR^a} \|v_s\|_{\delta}\;,
\end{equ}
for some $C > 1$.
\end{proof}

\begin{corollary}\label{cor:Linear}
    For $T>0$, $\beta\in\R$ and a continuous map $A:[0,T]\rightarrow \B(\CC^\beta ; \CC^{\beta
-1})$, denote by $(J_{s,t}^A)_{0\leq s\leq t \leq T}$ be the linear propagator of 
\begin{equ}
\d_t J = -|\nabla|^{2\gamma}J + A_tJ\;.
\end{equ}
Then for any $\theta\in[0,1)$ and $\delta$ such that
\begin{equ}
    \beta-1+2\gamma > \delta+2\gamma\theta> \beta\;,
\end{equ} 
for $\kappa > 0$ sufficiently small we have
\begin{equ}[e:JstBoundLip]
\|(J_{s,t}^A - J_{s,t}^B) v\|_{\delta + 2\gamma\theta}\lesssim C_{T,R} \sup_{\tau \in [0,T]}\|A_\tau-B_\tau\|_{\CC^{\beta} \rightarrow \CC^{\beta-1}} |t-s|^{1-\theta - \frac1{2\g} - \kappa} \|v\|_{\delta}\;,
\end{equ}
where $R=\sup_{t\in[0,T]} \|A_t\|_{\CC^{\beta} \rightarrow \CC^{\beta-1}} + \|B_t\|_{\CC^{\beta} \rightarrow \CC^{\beta-1}}$.
\end{corollary}
\begin{proof}
    Let $\wt J := J^A - J^B.$ Then $\wt J$ solves the equation
    \begin{align*}
        \partial_t \wt J = - |\nabla|^{2\g} \wt J + A_t J^A - B_t J^B = - |\nabla|^{2\g} \wt J + A_t \wt J + (A_t - B_t) J^B\;. 
    \end{align*}
    Therefore, by variation of constants, we have that 
    \begin{equ}
        \wt J_{s,t} = \int_s^t J^A_{\tau,t} (A_\tau - B_\tau) J^B_{s,\tau} \,d \tau.
    \end{equ}
    Taking $\kappa>0$ small enough and using that $\gamma>\frac12$, \eqref{e:JstBound} yields
    \begin{align*}
        &\| \wt J_{s,t}\|_{\CC^\dl \to \CC^{\dl + 2 \g \ta }} \\
        & \le \int_s^t \|J^A_{\tau,t}\|_{\CC^{\beta -1} \to \CC^{\dl + 2\g\ta}} \|A_\tau - B_\tau\|_{\CC^{\beta} \to \CC^{\beta -1}} \|J^B_{s,\tau}\|_{\CC^\dl \to \CC^\beta} \,d\tau \\
        &\les \sup_{\tau \in [0,T]}\|A_\tau-B_\tau\|_{\CC^{\beta} \rightarrow \CC^{\beta-1}}  \int_s^t (t-\tau)^{-\frac{\dl + 2\g \ta - \be +1}{2\g} }(\tau-s)^{-\frac{\beta - \dl}{2 \g} - \kappa} \,d \tau \\
        &\les \sup_{\tau \in [0,T]}\|A_\tau-B_\tau\|_{\CC^{\beta} \rightarrow \CC^{\beta-1}} |t-s|^{1-\theta - \frac1{2\g} - \kappa}\;,
    \end{align*}
    thus concluding the proof.
\end{proof}

\subsection{Strong Feller property}

We now show that the strong Feller property holds for a slightly larger class of
equations which also includes \eqref{e:equv}.
For this, we set
\begin{equ}[e:SNSgen2]
        \partial_t u+((|\nabla|^{\beta-\alpha} u)\cdot\nabla)u=-|\nabla|^{2\gamma} u-\nabla p+\boldsymbol{\xi}_{\beta + 1-\gamma}, \qquad
        \nabla\cdot u=0\;.
\end{equ}
We also re-write the modified equation \eqref{e:equv} in velocity formulation as
\begin{equ}[e:equvvelocity]
    \partial_t v+|\nabla|^{-\alpha}\bigl((v\cdot\nabla)|\nabla|^{\alpha}v\bigr)=-|\nabla|^{2\gamma} v-\nabla p+\boldsymbol{\xi}_{\alpha + 1-\gamma}, \qquad
        \nabla\cdot v=0\;.
\end{equ}
By setting $u=|\nabla|^\alpha v$ we see that \eqref{e:equvvelocity} corresponds to the case $\beta = 0$ while \eqref{e:SNSgen} 
corresponds to $\beta = \alpha$.

\begin{proposition}\label{prop:SF}
Let $\gamma \in (2/3,1]$, let $\alpha > 0$, let $\beta \in [0,\alpha]$, and let $\kappa \in (0,\beta)$. 
Then, the Markov semigroup on $\CC^{\beta-\kappa}$ generated
by the maximal solutions to \eqref{e:SNSgen2} satisfies the strong Feller property.
\end{proposition}

\begin{proof}
We would like to apply \cite[Thm~3.2]{SFGeneral}, so we need to verify Assumptions~1--5
in Section~2 of that article. In our setting, one has $H_0 = H^{\beta+1-\gamma}$ and we 
take for $\boldsymbol{\xi}$ the identity map with $\MM = \CC^{-1/2-\kappa/10}(\R,\CC^{\beta-\gamma-\kappa/10})$.
This is chosen in such a way that the space-time convolution with $P_t^{(\gamma)}$
is continuous from $\MM$ into $\CC(\R, \CC^{\beta-\kappa})$.

The solution map $\Phi_{s,t}(u_0,\xi)$ is then simply given by the local solution 
to \eqref{e:SNSgen2} on $U = \CC^{\beta-\kappa}$, which is well-defined by a 
standard Picard iteration. It is furthermore differentiable with respect to the initial 
condition $u_0$, so that Assumption~1 is satisfied. Assumption~2 is satisfied by
taking for $r$ the map $r_t(u_0,\xi) = \sup_{s \in [0,t]} \|\Phi_{0,s}(u_0,\xi)\|_{\beta-\kappa}$
with the understanding that $r_t(u_0,\xi) = +\infty$ if the solution has exploded before time $t$.
The regularity assumptions on $r$ can be verified in a straightforward way from the local
well-posedness of the solutions. Assumption~3 is trivially satisfied if we choose
$E = L^p([0,1],\CC^{\beta+1-\gamma})$ for any $p \ge 2$, with the action $\tau$ simply given 
by $\tau(\xi,h) = \xi+h$.
Assumption~4 just states that $h \mapsto \Phi_{s,t}(u_0,\xi+h)$ is Fréchet differentiable
in $h$ at $h=0$, provided that $u_0$ and $\xi$ are such that the solution hasn't exploded
yet by time $t$. This is straightforward to verify, with the derivative given by
\begin{equ}[e:MalDer]
D_h \Phi_{s,t}(u_s,\xi+\cdot) = \int_s^t J_{r,t} h_r\,dr\;.
\end{equ}
As a consequence of Lemma~\ref{lem:boundsLinear}, $J_{r,t}$ does indeed map $\CC^{\beta+1-\gamma}$
into $U = \CC^{\beta-\kappa}$ with bounds that are uniform over $s,t \in [0,1]$.

It remains to verify Assumption~5. Given $u_0 \in U$ and $\xi \in \MM$, we define
$A_t^{(s)}(u_0,\xi) \in L(U,E_s)$ by setting
\begin{equ}[e:defAt]
\bigl(A_t^{(s)}(u_0,\xi)v\bigr)(r) = \f2t \one( r\in [t/4,3t/4] \cap [0,s]) J_{0,r}v
\end{equ}
Since, for every $\kappa > 0$, $J_{0,t} v$ belongs to $\CC^{\beta + 2\gamma-1-\kappa}$, 
this does indeed belong to $E$, as long
as $\gamma > 2/3$. Furthermore, inserting the right-hand side of \eqref{e:defAt}
into \eqref{e:MalDer}, we have indeed
\begin{equ}
 \int_s^t J_{r,t} \bigl(A_t^{(t)}(u_0,\xi)v\bigr)(r)\,dr = J_{0,t}v\;,
\end{equ}
as required.
\end{proof}

\begin{remark}
There is a subtle point here that is worth remarking on. For small but non-zero $\beta$, we cannot rule out 
that \eqref{e:SNSgen2} admits global solutions that are unique in some natural way but that
are such that the $\CC^{\beta-\kappa}$-norm of the solution blows up at some finite time. In such a 
situation, one would have two Markov semigroups: the one where solutions are continued in their natural
way beyond their blow-up time and the one where, at blow-up time, the process gets killed (or 
equivalently moved to
a cemetery state which it never leaves). Proposition~\ref{prop:SF} only makes a claim about the latter 
semigroup! When $\beta > 2-\gamma$ or $\beta = 0$ we know that solutions are global (by enstrophy bounds in the
first case and having an explicit invariant measure in the second case), which is why 
Proposition~\ref{prop:SF} is mainly useful in that setting.
\end{remark}

A corollary of Proposition~\ref{prop:SF} is the following result which is 
a crucial ingredient in our approach.

\begin{proposition}\label{prop:equivalence}
Let $\gamma = 1$ and $\alpha > 0$ or $\gamma \in (\f23,1)$ and $\alpha> 2-\gamma$. 
Then \eqref{e:SNSgen} admits a unique invariant measure $\rho$. Furthermore, 
for every $u_0 \in \CC^{\alpha-\kappa}$ and every $t>0$, the transition probabilities 
$P_t(u_0, \cdot)$ are equivalent to $\rho$. 
\end{proposition}

\begin{proof}
Since we already know of the existence of $\rho$ by Lemmas~\ref{lem:goodBounds} and~\ref{lem:apriori} 
and that the semigroup
generated by the solutions to \eqref{e:SNS} satisfies the strong Feller property
by Proposition~\ref{prop:SF}, both statements
are a consequence of Khasminskii's theorem \cite[Prop.~4.1.1]{DPZ2} if we can show 
that the semigroup is topologically irreducible.

This in turn is a simple consequence of the continuity and surjectivity of the solution 
map in vorticity form, combined with the fact that the law of the noise $\xi_{\alpha-\gamma}$ has 
full support.
\end{proof}

We conclude this section with a priori bounds on the modified equation \eqref{e:equvvelocity}.

\begin{proposition} \label{prop:asGWPv}
Let $\gamma = 1$ and $\alpha > 0$, or $\gamma \in (\f23, 1)$ and $\alpha > 2 - \gamma$. Then the SPDE \eqref{e:equvvelocity} is globally well posed in $\CC^{\al-\kappa} $ for every $\kappa > 0$ small enough. 
Moreover, the measure $\mu_\al$ is invariant for the semigroup generated by this equation.
\end{proposition}

\begin{proof}
The construction of global stationary solution with fixed time marginals given by $\mu_\alpha$
is virtually identical to that given in \cite{DPD}, so we do not reproduce it here. The situation
given here is in fact easier since the equation is more regular and does not require any renormalisation.

This however only shows that one has global solutions for $\mu_\alpha$-every initial condition. 
The extension to every initial condition $\CC^{\al-\kappa}$ then follows from the strong Feller
property established in Proposition~\ref{prop:SF}.
\end{proof}

\section{Proof of the main theorems}
\label{sec:proof}

\subsection{An abstract equivalence criterion}

The goal of this subsection is to prove that under appropriate assumptions, the laws of $X(t), Y(t)$ are equivalent for every $t \ge 0$, where $X,Y$ are the solutions of the two SPDEs
\minilab{e:bothSPDEs}
\begin{equs}
dX &= -|\nabla|^{2\gamma} X\,dt + F(X)\,dt + |\nabla|^{\gamma-1-\alpha} dW\;, \label{e:abstractX}\\
dY &= -|\nabla|^{2\gamma} Y\,dt + F(Y)\,dt + G(Y)\,dt + |\nabla|^{\gamma-1-\alpha}dW\;, \label{e:abstractY}
\end{equs}
subject to the same initial data. More precisely, we aim to show the following 
\begin{proposition} \label{prop:equiv}
Consider the SPDEs \eqref{e:bothSPDEs}
with the same initial condition $X_0 = Y_0 \in \CC^{\al-\kappa}$, where $W$ is a cylindrical Wiener process on $L^2$. Assume that $\gamma\in(2/3,1]$, and for all $\kappa>0$ sufficiently small, $F$ is uniformly smooth
from $\CC^{\alpha-\kappa}$
to $\CC^{\alpha-1-\kappa}$, while $G$ is uniformly smooth
from $\CC^{\alpha-\kappa}$ to $\CC^{\alpha+1-2\gamma+\kappa}$.
Suppose moreover that both SPDEs admit a probabilistically strong  global solution
in the sense that for every $T>0$, almost-surely there exist $X,Y \in L^\infty([0,T]; \CC^{\alpha-\kappa})$ that solves the respective SPDE in mild formulation.\footnote{We remark that from the assumptions, such solutions must be unique.}
Then, for 
every $t>0$,
the laws of $X(t)$ and $Y(t)$ are equivalent.
\end{proposition}
At the core of the proof of this statement is an application of Girsanov's theorem to 
a suitable family of modified SPDEs interpolating between $X$ and $Y$. 
To this end, define $X^{(h,s)}_t$ as the solution to the SPDE
\begin{equ}[e:SPDEgen]
dX^{(h,s)}_t = -|\nabla|^{2\gamma} X^{(h,s)}_t\,dt + F_s(X^{(h,s)}_t)\,dt + |\nabla|^{\gamma-1-\alpha} (dW + h\,dt)\;,
\end{equ}
where $W$ is a cylindrical Wiener process on $L^2$, and $h$ is a priori just a  measurable drift. We will assume the following 
on the family of equations \eqref{e:SPDEgen}. 
\begin{assumption} \label{a:SPDEgen1}
    For all $\kappa>0$ sufficiently small, $(F_s)_{s\in[0,1]}$ is a uniformly smooth path from $\CC^{\alpha-\kappa}$ to $\CC^{\alpha-1-\kappa}$. Furthermore, we assume that $\d_s F_s$ maps $\CC^{\alpha-\kappa}$ to $\CC^{\alpha+1-2\gamma + \kappa}$ and is Lipschitz on bounded sets.
\end{assumption}
\begin{assumption} \label{a:SPDEgen2}
The equation \eqref{e:SPDEgen} has global solutions, in the sense that for every $T>0$, 
for every  $h \in L^2([0,T], L^2)$, and for every initial data $X_0 \in \CC^{\al-\kappa}$,
    $$ \sup_{s\in[0,1]}\sup_{t \in [0,T]} \|X^{(h,s)}_t\|_{\CC^{\al - \kappa}} < \infty \quad \text{a.s.} $$
\end{assumption}
\begin{assumption} \label{a:SPDEgen3}
    The map $ L^2([0,T]; L^2) \ni h \to X^{(h,s)}_t \in L^\infty([0,T]; \CC^{\al - \kappa})$ is locally Lipschitz uniformly in $s$. More precisely, for every $R >0$, we have that 
    \begin{equ}
        \sup_{\|h_1\|+ \|h_2\| \le R}\, \sup_{s\in [0,1]} \frac{\|X^{(h_1,s)}_t - X^{(h_2,s)}_t\|_{L^\infty([0,T]; \CC^{\al - \kappa})}}{\|h_1 - h_2\|_{L^2([0,T]; L^2)}} < \infty \quad \text{a.s.}
    \end{equ}
\end{assumption}

We also write $J^{(h,s)}$ for the solution to the linearised equation, namely
\begin{equ}\label{e:JstSPDE}
\d_t J^{(h,s)}_{r,t} = -|\nabla|^{2\gamma} J^{(h,s)}_{r,t} + DF_s(X^{(h,s)}_t)J^{(h,s)}_{r,t}\;,
\quad J^{(h,s)}_{r,r} = \id\;.
\end{equ}
In terms of the notations of Lemma~\ref{lem:boundsLinear}, we have
\begin{equation}
    J^{(h,s)}=J^{DF_s(X^{(h,s)})}\;.
\end{equation}
The idea now is to find a map $s \mapsto h_s$ such that, for some fixed final time $T$,
the solution $X^{(h_s,s)}_T$ is independent of $s$. A simple calculation shows that
\begin{equ}[e:dsXscalculation]
\d_s X^{(h_s,s)}_T = \int_0^T J^{(h_s,s)}_{r,T} \bigl((\d_s F_s)(X^{(h_s,s)}_r) + |\nabla|^{\gamma-1-\alpha} \d_s h_s(r)\bigr)\,dr\;.
\end{equ}
Proceeding as in \eqref{e:SGTrick}, the term including $\d_s F_s$ can be rewritten as
\begin{equ}
 \int_0^T J^{(h_s,s)}_{r,T} \bigl(\d_s F_s(X^{(h_s,s)}_r)\bigr)\,dr =
 2\int_{T/2}^T J^{(h_s,s)}_{r,T} \bigl(J^{(h_s,s)}_{2r-T,r}\d_s F_s(X^{(h_s,s)}_{2r-T})\bigr)\,dr\;.
\end{equ}
It follows that, for $X^{(h_s,s)}_T$ to be independent of $s$, we can choose $h$ 
such that $h_s(\tau) = 0$ for $\tau < T/2$ and, for $\tau > T/2$,
\begin{equ}[e:shiftEquation]
  \partial_s h_s(\tau) = -2 |\nabla|^{\al+1-\g}  J_{2\tau - T, \tau}^{(h_s,s)} [\d_s F_s( X^{(h_s,s)}_{2\tau - T}) ]\;.
\end{equ}
We shall seek solutions to \eqref{e:shiftEquation} in the space
\begin{equation}
    \mathbb{Y}=\CC^0([0,1];L^2_TL^2_x)\;,
\end{equation}
namely the space of continuous paths taking values in the Cameron--Martin space of space-time white noise.
As it turns out, in order to solve this ODE, it is convenient to truncate \eqref{e:shiftEquation} in order to guarantee existence of a solution up to the target time $T$.
\begin{lemma} \label{lem:randomODE}
    For $\gamma\in(2/3,1]$ and $\kappa$ sufficiently small, let $\chi: \CC^{\al - \kappa} \to \R$ be a Lipschitz function such that $\supp(\chi) \subseteq \{ x \in \CC^{\al - \kappa}: \| x \|_{\CC^{\al - \kappa}} \le 2\}$. Let $R > 0$. Then, there almost  surely  exists $h\in\mathbb{Y}$ solving the random ODE  
    \begin{equ}[e:shiftEquationTruncated]
        \partial_s h_s(\tau) = -2 \mathbbm{1}_{[T/2,T]}(\tau) |\nabla|^{\al+1-\g}  J_{2\tau - T, \tau}^{(h_s,s)} [\d_s F_s( X^{(h_s,s)}_{2\tau - T}) ] \chi(X^{(h_s,s)}_{2\tau - T}/R)\;
    \end{equ}
    with initial data $h_0 = 0$.
    Moreover, there exists $\eps>0$ and, for every $R>0$ a constant $C_R$, such that $h$ satisfies the almost sure bounds 
    \begin{gather} \label{e:shiftBound1}
             \|h_s\|_{\CC^1([0,1];L^2_TL^2)}\leq T^\varepsilon C_R\;,
    \end{gather} 
uniformly over $T \in [0,1]$.
    Finally, for each $s$, $h_s$ is progressively measurable with respect to the filtration generated by $dW$.
\end{lemma}
\begin{proof}
    For any fixed realisation of the noise, we solve \eqref{e:shiftEquationTruncated} by a standard Picard iteration.
We write the right-hand side of \eqref{e:shiftEquationTruncated} as $\mathbb{F}_r:L^\infty_T \CC^{\al-\kappa}\rightarrow L^2_TL^2$ given by
    \begin{equation*}
        \mathbb{F}_r(X)(\tau)=\mathbbm{1}_{[T/2,T]}(\tau) |\nabla|^{\al+1-\g}  J_{2\tau - T, \tau}^{DF_r(X)} [\d_r F_r( X_{2\tau - T}) ] \chi(X_{2\tau - T}/R)\;.
    \end{equation*} By Assumption \ref{a:SPDEgen3}, the map $$ L^2_TL^2_x \ni h \mapsto X^{(h,s)} \in L^\infty([0,T], \CC^{\al-\kappa})$$
    is Lipschitz on bounded sets, so we only need to show that $\mathbb{F}_s$ maps boundedly into $L^2_TL^2$, and is Lipschitz (uniformly in $s$). To see the boundedness, we set $\beta=\alpha-\kappa$ and see
    \begin{equs}
        \sup_{X\in L^\infty_T\CC^{\beta}}&\| \mathbb{F}_s(X)\|_{L^2_TL^2}^2 \\
        &\lesssim \sup_{\|X\|_{L^\infty_T\CC^{\beta}}\leq 2R}\int_{\frac T2}^T\left\|J_{2\tau - T, \tau}^{DF_s(X)} [\d_s F_s( X_{2\tau - T}) ]\right\|^2_{H^{\alpha+1-\gamma}}\,d\tau \\
        &\lesssim \sup_{\|Y\|_{\CC^{\beta}}\leq 2R}\left\|\d_s F_s(Y) \right\|^2_{\CC^{\delta}}\sup_{\|X\|_{L^\infty_T\CC^{\beta}}\leq 2R}\int_{\frac T2}^T\left\|J_{2\tau - T, \tau}^{DF_s(X)}\right\|^2_{\CC^{\delta}\rightarrow \CC^{\delta+2\gamma\theta}} \,d\tau\;,
    \end{equs} for any $\delta,\theta$ such that
    \begin{equation}\label{e:deltagamma1}
        \delta+2\gamma\theta>\alpha+1-\gamma\;.
    \end{equation} 
    We choose $\delta=\alpha+1-2\gamma+\kappa$, so by Assumption \ref{a:SPDEgen1} we have
    \begin{equation*}
        D_{R}= \sup_{s\in[0,1]}\sup_{\|Y\|_{\CC^{\beta}}\leq 2R}\left\|\d_s F_s(Y) \right\|_{\CC^{\delta}}<\infty\;.
    \end{equation*}
    For this choice of $\delta$, if $\theta\in[0,\frac12)$ such that
    \begin{equation}\label{e:deltagamma2}
        \beta-1+2\gamma>\delta+2\gamma\theta>\beta,
    \end{equation}
    then by \eqref{e:JstBound} there exists a constant $\tilde C_R$ such that
    \begin{align*}
        \left\|J_{2\tau - T, \tau}^{DF_s(X)}\right\|_{\CC^{\delta}\rightarrow \CC^{\delta+2\gamma\theta}}\leq \tilde C_R|T-\tau|^{-\theta}\;,
    \end{align*}
uniformly over $T \in [0,1]$.
    So if there is a choice of $\theta\in[0,\frac12)$ which satisfies \eqref{e:deltagamma1} and \eqref{e:deltagamma2} then we have
    \begin{equs}[e:boundFs]
        \sup_{X\in L^\infty_T\CC^{\beta}}\| \mathbb{F}_s(X)\|_{L^2_TL^2}
        &\lesssim \wt{C}_{T,Q_R}D_R\Bigl(\int_{T/2}^T|T-\tau|^{-2\theta}\,d\tau\Bigr)^{\frac12} \\
        &\lesssim T^\varepsilon C_{T,Q_R}D_R<\infty\;,
    \end{equs}
    for 
    \begin{equation*}
        Q_{R}= \sup_{s\in[0,1]}\sup_{\|Y\|_{\CC^{\beta}}\leq 2R}\left\|DF_s(Y)\right\|_{\CC^{\beta}\rightarrow \CC^{\beta-1}}<\infty\;.
    \end{equation*}
    We see that \eqref{e:deltagamma1} is equivalent to
    $
        \theta>\frac12-\frac{\kappa}{2\gamma}
    $
    and \eqref{e:deltagamma2} is equivalent to
    \begin{equation*}
        2-\tfrac{1}{\gamma}-\tfrac{\kappa}{\gamma}>\theta>1-\tfrac{1}{2\gamma}-\tfrac\kappa\gamma\;.
    \end{equation*}
    This can be achieved for some $\theta<\frac12$ if
    \begin{equation*}
        2-\tfrac1\gamma>\tfrac12\geq 1-\tfrac{1}{2\gamma}\quad\iff\quad \gamma\in(2/3,1]\;,
    \end{equation*} and $\kappa>0$ is sufficiently small (depending on $\gamma$). Thus $\mathbb{F}_s$ maps boundedly into $L^2_TL^2$. 
    
    To show that $\mathbb{F}_s$ is Lipschitz, by the cutoff and the boundedness of $\mathbb{F}_s$, we reduce to bounding
    \begin{equation*}
        \left\|J_{2\tau - T, \tau}^{DF_s(X)} \bigl[\d_s F_s( X_{2\tau - T}) \bigr] - J_{2\tau - T, \tau}^{DF_s(Y)} \bigl[\d_s F_s( Y_{2\tau - T}) \bigr]\right\|_{L^2([T/2,T];H^{\alpha+1-\gamma})}\;,
    \end{equation*}
    for $\|X\|_{L^\infty_T\CC^{\beta}},\,\|Y\|_{L^\infty_T\CC^{\beta}}\leq 2R$. We separately bound
    \begin{equation*}
        \left\|J_{2\tau - T, \tau}^{DF_s(X)} \bigl[\d_s F_s( X_{2\tau - T}) -\d_s F_s( Y_{2\tau - T}) \bigr]\right\|_{L^2([T/2,T];H^{\alpha+1-\gamma})}\;,
    \end{equation*}
    and
     \begin{equation*}
        \left\|\bigl(J_{2\tau - T, \tau}^{DF_s(X)} - J_{2\tau - T, \tau}^{DF_s(Y)}\bigr) \bigl[\d_s F_s( Y_{2\tau - T}) \bigr]\right\|_{L^2([T/2,T];H^{\alpha+1-\gamma})}\;.
    \end{equation*}
    The first term is bounded exactly as above, using Assumption~\ref{a:SPDEgen1} that $\d_s F_s$ is locally Lipschitz. For the second term we take the same approach as before, namely
    \begin{align*}
        &\left\|\bigl(J_{2\tau - T, \tau}^{DF_s(X)} - J_{2\tau - T, \tau}^{DF_s(Y)}\bigr) \bigl[\d_s F_s( Y_{2\tau - T}) \bigr]\right\|_{L^2([T/2,T];H^{\alpha+1-\gamma})} \\
        &\lesssim D_R\left\|J_{2\tau - T, \tau}^{DF_s(X)} - J_{2\tau - T, \tau}^{DF_s(Y)}\right\|_{L^2([T/2,T];\CC^{\delta}\rightarrow \CC^{\delta+2\gamma\theta})}\;,
    \end{align*}
    and by \eqref{e:JstBoundLip} we achieve the Lipschitz bound (note the condition $\theta<\frac12$ is replaced by the condition $1-\theta-\frac{1}{2\gamma}-\kappa>-\frac12$, but for $\gamma>\frac12$, this is a weaker condition on $\theta$).
    So by a standard fixed point argument (iterating small intervals of $s$), there is a unique solution to \eqref{e:shiftEquationTruncated}, almost-surely. Moreover, the bound on $\mathbb{F}_s$ given in \eqref{e:boundFs}
    implies the desired uniform bound \eqref{e:shiftBound1}.
 
To show that $h$ is progressively measurable we use the following fact. Let $h$ solve the 
ODE $\d_s h = \mathbf{F}_s(h)$ in some Banach space $B$, where $\mathbf{F}_s$ is locally Lipschitz
continuous, uniformly over $s \in [0,1]$. Let furthermore $\Pi \colon B \to B$ be a continuous linear map
such that 
\begin{equ}[e:restrProp]
\Pi \mathbf{F}_s(h) = \Pi\mathbf{F}_s(\Pi h)\;,
\end{equ}
for every $s$ and every $h \in B$.
Then $\tilde h = \Pi h$ is the unique solution to the ODE $\d_s \tilde h  = \Pi \mathbf{F}_s(\tilde h)$.
It then suffices to note that, for every $t \in [0,T]$, the ``restriction map''
$\Pi_t h = \one_{[0,t]}h$ is indeed such that \eqref{e:restrProp} holds for
\begin{equ}
\mathbf{F}_s(h) = \mathbb{F}_s(X^{(h,s)})\;.
\end{equ}
It is furthermore straightforward from the definitions that
$\Pi_t \mathbf{F}_s$ is $\CF_t$-measurable, whence we deduce that 
the same is true of $\Pi_t h$. Since $h$ is continuous on $[T/2,T]$, the claim follows.
\end{proof}

\begin{proof}[Proof of Proposition \ref{prop:equiv}]
    We show that $\Law(X(t)) \ll \Law(Y(t))$. The reverse absolute continuity follows from replacing $F \mapsto F+G$, $G\mapsto -G$. For convenience of notation, denote by $Y^h$ the solution to 
    \begin{equ}
        dY = -|\nabla|^{2\gamma} Y\,dt + F(Y)\,dt + G(Y)\,dt + |\nabla|^{\gamma-1-\alpha}(dW + h\, dt)\;,
    \end{equ}
    and write $Z$ for the Ornstein--Uhlenbeck process 
    \begin{equ}
        dZ = -|\nabla|^{2\gamma} Z\,dt + |\nabla|^{\gamma-1-\alpha}dW\;, \qquad
        Z(0) = X_0\;.
    \end{equ}
    For $R\ge 0$ and $T>0$ we let $F_{R,T}$ be the event
    \begin{equ}
        F_{R,T}=\Bigl\{ \sup_{\tau \in [0,T]} \bigl(\|Z_\tau \|_{\CC^{\al - \kappa}} + \|X_\tau \|_{\CC^{\al - \kappa}}\bigr) < R\Bigr\}.
    \end{equ}
    For a random variable $X$ and an event $E$, we introduce the notation 
    \begin{equ}
        \Law(X;E)
    \end{equ}
    to denote the restriction of the law of $X$ to the event $E$.
    In particular, one has $\Law(X;E) \to \Law(X)$ in total variation as $\P(E) \to 1$.
    
    We first aim to show that for all $R>\|X_0\|_{\CC^{\alpha-\kappa}}$, there is a $\wt{T}_R>0$ such that for all $0<t<\wt{T}_R$, we have
    \begin{equ}[e:localac]
        \Law(X(t);F_{R,t}) \ll \Law(Y(t)).
    \end{equ}
    (Note here that the event $F_{R,t}$ does not just depend on $X(t)$ but on the whole trajectory
    up to time $t$.) We then will conclude by using the strong Markov property and global-well-posedness of solutions.
    To obtain \eqref{e:localac}, we introduce two cutoffs: a Lipschitz cutoff for the $\CC^{\alpha-\kappa}$-norm to apply Lemma \ref{lem:randomODE}, and a uniformly smooth cutoff in the nonlinearity controlling a weaker norm.
    For an appropriate $\chi_0: \R \to \R$ smooth with compact support and $p = p(\al,\kappa)$ an even integer sufficiently large, let $\chi,\chi^{\rm{sm}}: \CC^{\alpha - \kappa} \to \R$ be given by 
    \begin{equ}
        \chi(X) = \chi_0(\|X\|_{\CC^{\alpha-\kappa}})\;, \qquad
        \chi^{\rm sm}(X) = \chi_0(\|X\|_{B^{\al - 2\kappa}_{p,p}}^p)\;.
    \end{equ}
    We note that $\chi$ is Lipschitz with bounded support, and $\chi^{\rm sm}$ is smooth with locally bounded derivatives (due to the smoothness of the $B^{\al - 2\kappa}_{p,p}$-norm for $p$ an even integer, and the embedding $\CC^{\alpha-\kappa}\hookrightarrow B^{\al - 2\kappa}_{p,p}$), but with unbounded support\footnote{By \cite[Theorem 25]{Fry02}, there is no bump function with Lipschitz Fr\'echet derivative on~$\CC^{\alpha-\kappa}$.}. For any choice of $p$, we may choose $\chi_0$ so that $\chi(X)=\chi^{\rm{sm}}(X)=1$ for $\|X\|_{\CC^{\al -\kappa}} \le 1$. For $p$ sufficiently large, we further have the embedding
    \begin{equation*}
        B^{\alpha-2\kappa}_{p,p}\hookrightarrow \CC^{\alpha-3\kappa},
    \end{equation*}
    and so the support of $\chi^{\rm{sm}}$ is bounded in the $\CC^{\alpha-3\kappa}$-norm. For $R>0$ we let
    \begin{equation*}
        \chi_R(X)=\chi(X/R)\;,\qquad \chi^{\rm{sm}}_R(X)=\chi^{\rm{sm}}(X/R)\;.
    \end{equation*}
    For fixed $R\gg \|X_0\|_{\CC^{\alpha-\kappa}} \gtrsim \|X_0\|_{\CC^{\al - 3\kappa}}$,  we let
    \begin{equ}
        F_s = F \chi^{\rm sm}_{2R} + sG \chi^{\rm sm}_{2R}\;.
    \end{equ}
    We see that $F_s$ is a uniformly smooth path, and satisfies Assumption \ref{a:SPDEgen1}. To see that Assumption \ref{a:SPDEgen2} is satisfied, we see for $0<2\kappa<2\gamma-1$, and $3\kappa$ small enough (depending on $F$),
    \begin{align*}
        \left\|\int_0^t P_{t-\tau}^{(\gamma)} F(X_\tau)\chi^{\rm sm}_{2R}(X_\tau)\,d\tau\right\|_{\CC^{\alpha-\gamma}}&\lesssim \|F(X_\tau)\chi^{\rm sm}_{2R}(X_\tau)\|_{L^\infty_t\CC^{\alpha-\kappa-2\gamma}} \\
        &\lesssim \sup_{\|X\|_{\CC^{\alpha-3\kappa}}\lesssim R}\|F(X)\|_{\CC^{\alpha-1-3\kappa}} \\
        &\lesssim C_{R,\kappa}<\infty\;,
    \end{align*} namely that the $\CC^{\alpha-3\kappa}$-norm controls the growth of solutions. 
    Iterating the fixed point argument, we see that Assumption~\ref{a:SPDEgen2} holds, and Assumption~\ref{a:SPDEgen3} follows suit. Omitting the dependence on $R$, we let $\wt{X}^{(h,s)}$ denote the solution to 
    \begin{equ}
        d\wt{X}^{(h,s)}_t = -|\nabla|^{2\gamma} \wt{X}^{(h,s)}_t\,dt + F_s(\wt{X}^{(h,s)}_t)\,dt + |\nabla|^{\gamma-1-\alpha} (dW + h\,dt)\;.
    \end{equ}
    Fix $T_R>0$ to be determined later and let $h\in\CC^1([0,1];L^2_{T_R}L^2)$ be the solution from Lemma~\ref{lem:randomODE} to the random ODE
    \begin{equ}
        \partial_s h_s(\tau) = -2 \mathbbm{1}_{[T_R/2,T_R]}(\tau) |\nabla|^{\al+1-\g}  J_{2\tau - T_R, \tau}^{(h_s,s)} [\d_s F_s( \wt{X}^{(h_s,s)}_{2\tau - T_R}) ] \chi_{2R}(\wt{X}^{(h_s,s)}_{2\tau - T_R})\;.
    \end{equ}
    We note that $h$ depends on the choice of $R$ and $T_R$.
    On the event
    \begin{equ} \label{e:goodEvent0}
        E_{R,T_R} = \Bigl\{ \sup_{s\in[0,1]}\sup_{\tau \in [0,T_R]} \|\wt{X}^{(h_s,s)}_\tau \|_{\CC^{\al - \kappa}} < 2R\Bigr\}\;,
    \end{equ}
    we necessarily have $\sup_{\tau \in [0,T]} \|\wt{X}^{(0,0)}_\tau \|_{\CC^{\al - \kappa}},\|\wt{X}^{(h_1,1)}_\tau \|_{\CC^{\al - \kappa}}<2R$. On this event, we have that $F_0=F$ and $F_1=F+G$, hence for all $\tau\in[0,T_R]$ we have that
    \begin{equ}
        X_{\tau}=\wt{X}_{\tau}^{(0,0)},\quad Y_\tau^{h_1}=\wt{X}_{\tau}^{(h_1,1)}\;.
    \end{equ}
    Moreover, on this event $h$ satisfies \eqref{e:shiftEquation}, and so
    \begin{equ}[e:TRendpointequal]
        X_{T_R} = \wt{X}_{T_R}^{(0,0)} = \wt{X}_{T_R}^{(h_1,1)} =  Y_{T_R}^{h_1}\;.
    \end{equ}
    In light of the uniform bound \eqref{e:shiftBound1}, $h_1$ is a bounded progressively measurable process in $L^2_{T_R}L^2$. In particular it satisfies Novikov's condition, so by Girsanov's theorem, we deduce
    \begin{equ}
        \Law(Y^{h_1}(T_R)) \ll \Law(Y(T_R))\;.
    \end{equ}
    By \eqref{e:TRendpointequal}, we see that
    \begin{equ}
        \Law(X_{T_R};E_{R,T_R})=\Law(Y^{h_1}(T_R);E_{R,T_R}) \ll \Law(Y(T_R))\;,
    \end{equ}
    so that \eqref{e:localac} follows if we can show that
       $F_{R,T_R}\subseteq E_{R,T_R}$
    for any $T_R$ sufficiently small. To this end, we see in mild formulation 
    \begin{equ}
        \wt{X}_t^{(h_s,s)}=Z_t+\int_0^t P^{(\gamma)}_{t-\tau}\bigl( F_s(\wt{X}_\tau^{(h_s,s)})+|\nabla|^{\gamma-1-\alpha} h_s(\tau)\bigr)d\tau\;.
    \end{equ}
    We recall \eqref{e:shiftBound1} implies
    \begin{align*}
        \left\|\int_0^t P^{(\gamma)}_{t-\tau}|\nabla|^{\gamma-1-\alpha} h_s(\tau)d\tau\right\|_{L^\infty([0,T_R] ; \CC^{\alpha-\kappa)}}&\lesssim \left\||\nabla|^{\gamma-1-\alpha} h_s\right\|_{L^2([0,T_R] ; H^{\alpha-\gamma-\kappa)}} \\
        &\lesssim T_R^{\varepsilon}C_R\;,
    \end{align*}
    and exactly as in the calculation verifying Assumption \ref{a:SPDEgen2}
    \begin{equs}
        \Bigl\|&\int_0^t P^{(\gamma)}_{t-\tau}F_s(\wt{X}_\tau^{(h_s,s)})d\tau\Bigr\|_{L^\infty([0,T_R] ; \CC^{\alpha-\kappa)}} \\ &\lesssim T_R^{\varepsilon}\sup_{X\in \CC^{\alpha-\kappa}} \|(F(X)+sG(X))\|_{\CC^{\alpha-\kappa-2\gamma(1-\varepsilon)}}|\chi^{\rm{sm}}_R(X)| \\
        &\lesssim T_R^{\varepsilon}\sup_{\|X\|_{\CC^{\alpha-3\kappa}}\leq CR} \|(F(X)+sG(X))\|_{\CC^{\alpha-\kappa-2\gamma(1-\varepsilon)}} \lesssim T_R^{\varepsilon} C_R,
    \end{equs}
    for $\alpha-3\kappa-1>\alpha-\kappa-2\gamma(1-\varepsilon)$, namely for $\kappa,\varepsilon$ sufficiently small (and $\gamma>1/2$). In particular, there exists some constant $C_R$ such that for all $T_R\ll 1$, on the event $F_{R,T_R}$ we have
    \begin{align*}
        \sup_{s\in[0,1]}\sup_{\tau \in [0,T_R]} \|\wt{X}^{(h_s,s)}_\tau \|_{\CC^{\al - \kappa}} &\leq \sup_{\tau\in[0,T]}\|Z_\tau\|_{}+T_R^{\varepsilon}C_R \\
        &\leq R+T_R^{\varepsilon}C_R,
    \end{align*}
    so that $F_{R,T_R}\subseteq E_{R,T_R}$ for $T_R^{\varepsilon}C_R< R$.
    It follows that for all $R>\|X_0\|_{\CC^{\alpha-\kappa}}$, there is a $\wt{T}_R>0$ such that for all $0<t<\wt{T}_R$, we have
    \begin{equ}
        \Law(X(t);F_{R,t}) \ll \Law(Y(t))\;.
    \end{equ}
    As this holds for any initial data and $X,Y$ have the strong Markov property by Remark~\ref{rem:SM}, we may iterate and obtain that for every $0 \le t_1, t_2 < \wt{T}_R $, we have that 
    \begin{equ}
        \Law(X(t_1 + t_2);F_{R,t_1} \cap F_{R,t_2}^{t_1}) \ll \Law(Y(t_1+t_2))\;,
    \end{equ}
    where $F_{R,t_2}^{t_1}$ denotes
    \begin{equ}
       F_{R,t_2}^{t_1} = \Bigl\{\sup_{\tau \in [t_1,t_1+t_2]} \bigl(\|Z^{t_1}_{\tau}\|_{\CC^{\alpha-\kappa}} + \|X_\tau \|_{\CC^{\al - \kappa}}\bigr) < R \Bigr\}\;,
    \end{equ}
    where $Z^{t_1}_\tau$ is the OU process restarted with initial condition $Z_{t_1}=X_{t_1}$.
    For a general $T>0$, we iterate this $n \ge \frac T {\wt{T}_R}$ times and obtain that
    \begin{equ}
        \Law(X(T);\tilde{F}_{R,T}) \ll \Law(Y(T))\;,
    \end{equ}
    where $T=t_1+\ldots +t_n $, for $0\leq t_i<\wt{T}_R$ and
    \begin{equ}
        \tilde{F}_{R,T}=F_{R,t_1}\cap F_{R,t_1}^{t_{2}}\cap\ldots\cap F_{R,t_{n-1}}^{t_{n}}\;.
    \end{equ}
    We note that $Z^{t_i}_\tau=Z_{\tau}+P^{(\gamma)}_{\tau-t_1}(X_{t_i}-Z_{t_i})$, and so $F_{R/3,T}\subseteq \tilde{F}_{R,T}$.
    By the global well-posedeness of $X$, we have that $\ind_{F_{R/3,T}} \to \ind$ a.s.\ as $R \to \infty$, so by taking limits in total variation we conclude that
       $$ \Law(X(T)) \ll \Law(Y(T))\;,$$
       thus completing the proof.
\end{proof}

\subsection{Application to stochastic Navier--Stokes equations}
We wish to apply Proposition \ref{prop:equiv} to equations \eqref{e:SNSgen} and \eqref{e:equvvelocity}, namely
\begin{equs}
     \partial_t X&=-|\nabla|^{2\gamma} X- P\div \bigl(X \otimes X\bigr) + \sqrt2\boldsymbol{\xi}_{\alpha +1 -\gamma}\;, \label{e:equvX}\\
    \partial_t Y&=-|\nabla|^{2\gamma} Y- |\nabla|^{-\alpha} P\mathrm{div} \bigl(Y \otimes |\nabla|^{\alpha}Y\bigr) + \sqrt2\boldsymbol{\xi}_{\alpha +1 -\gamma} \label{e:equvY}\;.
\end{equs}
These can be written as special cases of \eqref{e:abstractX} and \eqref{e:abstractY} respectively, with 
\begin{equs}
    F(u) &= P \mathrm{div} \bigl(u \otimes u\bigr)\;,\\
    G(u) &= |\nabla|^{-\alpha} P\mathrm{div} \bigl(u \otimes |\nabla|^{\alpha}u\bigr) -  P\text{div} \bigl(u \otimes u \bigr)\;.
\end{equs} 
For clarity, we make the dependence of $G$ on $\al$ explicit and write $G=G_\al$. In view of Proposition \ref{prop:equiv}, we first show that $G_\al$ is a uniformly smooth map from $\CC^{\al-\kappa} \to \CC^{\al + 1 - 2\g + \kappa}$. 
To do this, we first recall the following formulation of the Coifman--Meyer multiplier theorem (see for example \cite{CK}).
\begin{proposition}\label{prop:CM}
Let $m:\bb{R}^d\times\R^d\to \bb{C}$ be smooth away from the origin such that
\begin{equation*}
\|m\|_{\CM,n}\eqdef \sup_{0 \le n_1,n_2 \le n}\sup_{\xi_1,\xi_2\in\bb{R}^d}(|\xi_1|+|\xi_2|)^{n_1+n_2}|\nabla^{n_1}_{\xi_1}\nabla^{n_2}_{\xi_2} m(\xi_1,\xi_2)|<\infty\;,
\end{equation*}
for every $n \ge 0$.
Consider the bilinear map given by
\begin{equation*}
\mathcal{F}[T_m(f_1,f_2)](n)=\sum_{\substack{n_1,n_2\in\bb{Z}^2 \\ n_1+n_2=n}}m(n_1,n_2)\widehat{f_1}(n_1)\widehat{f_2}(n_2)\;.
\end{equation*}
Then for all $p\in[1,\infty)$ and $p_1,p_2\in(1,\infty]$ with $\frac{1}{p}=\frac{1}{p_1}+\frac{1}{p_2}$
there exists $n > 0$ such that
\begin{equation*}
\|T_m(f_1,f_2)\|_{L^p}\lesssim_{d,p_i}\|m\|_{\CM,n}\|f_1\|_{L^{p_1}}\|f_2\|_{L^{p_2}}\;,
\end{equation*}
 for all $f_1,f_2$.
\end{proposition} 

With this, we can show the desired mapping properties of $G_\alpha$.
\begin{proposition} \label{prop:GalRegularity}
    For any $\alpha>0$, $\beta>\alpha/2$, and $\kappa>0$, we have
\begin{equation}
    \|G_\alpha(u)\|_{\CC^{\min\{\beta,2\beta-1\}-\kappa}}\lesssim\|u\|_{\CC^{\beta}}^2.
\end{equation}
for any divergence-free $u$. In particular $G_\alpha$ maps $C^{\min\{\beta,2\beta-1\}-\kappa}$ to $C^\be$, and this is uniformly smooth.
\end{proposition}
\begin{proof}
    Since $G_\al$ is bilinear, it is enough to show that the operator  
    \begin{equation}
    G_\al(f,g)=\text{div} \bigl(f \otimes g \bigr) - |\nabla|^{-\alpha} \text{div} \bigl(f \otimes |\nabla|^{\alpha}g \bigr)\;,
    \end{equation}
    is bounded from $C^{\min\{\beta,2\beta-1\}-\kappa}$ to $C^\be$.
    For mean-free $f$, we recall that  
    \begin{equation*}
    \|Pf\|_{\CC^s}\lesssim \|f\|_{\CC^s}\;.
    \end{equation*}
    We write $G_\al$ as the bilinear Fourier multiplier given by
    \begin{equation}
        \F[G(f,g)](n)=i\sum_{k+\ell=n}|n|^{-\alpha}(|n|^{\alpha}-|\ell|^{\alpha})(\ft{f}(k)\cdot\ell)\ft{g}(\ell)\;.
    \end{equation}
    We decompose into different frequency scales; fix $(P_N)$ with $N$ dyadic integers a family of Littlewood--Paley projectors, namely $\F[P_Nf](n)=\phi_N(n)\ft{f}(n)$, where $\phi_N(n)=\phi(n/N)$ for $\phi$ a smooth function on $\R^2$ supported on an annulus $\frac12<|x|<2$ with $\sum_{N\text{ dyadic}} \phi_N(x)=1$ for all $|x| \ge 1$. We also define $\phi_{\approx N}=\sum_{N/8\leq M\leq 8N}\phi_M$.
    
    For $N,K,L$ dyadic we shall estimate $P_NG(P_Kf,P_Lg)$, computing
    \begin{equation}\label{dyadicbilinear}
        \F[P_NG(P_Kf,P_Lg)]
        (n)=i\sum_{k+\ell=n}\phi_N(n)|n|^{-\alpha}(|n|^{\alpha}-|\ell|^{\alpha})(\ft{P_K f}(k)\cdot\ell)\ft{P_L g}(\ell).
    \end{equation}
    In all the following we assume $|k|\approx K,|\ell|\approx L,|n|\approx N$ and $k+\ell=n$. We split into cases depending on the relative sizes of $K,L,N$. 
    
    For $K\ll L\approx N$  we may replace the multiplier by a localised version. Letting
    \begin{equation*}
    \psi_{N,\gamma}(x)=N^{-\gamma}|x|^\gamma\phi_{\approx N}(x)\;,
	\end{equation*}    then $\psi_{N,\gamma}(x)=\psi_{\gamma}(x/N)$ for some smooth compactly supported $\psi_{\gamma}$, and we have (for $K\ll L\approx N$)
    \begin{align*}
        \phi_N(n)|n|^{-\alpha}(|n|^{\alpha}-|\ell|^{\alpha})&=\phi_N(k+\ell)\psi_{N,-\alpha}(k+\ell)(\psi_{N,\alpha}(k+\ell)-\psi_{N,\alpha}(\ell)) \\
        &=\phi_N(k+\ell)\int_0^1N^{-1}k\cdot\nabla\psi_\alpha\left(\frac{\ell+\sigma k}{N}\right)d\sigma\;.
    \end{align*}
    So for $j=1,2$, letting
    \begin{equation*}
       m_N^j(k,\ell)= \phi_N(k+\ell)\phi_{\approx N}(\ell)\int_0^1\nabla_j\psi_\alpha\left(\frac{\ell+\sigma k}{N}\right)d\sigma\;,
    \end{equation*}
    then $m_N^j$ is a Coifman--Meyer multiplier with $\|m_N^j\|_{\CM,n}\lesssim 1$ for every $n$ (as $m_N^j$ is only non-zero for $|k|,|\ell|\lesssim N$, and each derivative of $m_N^j$ gains a factor $N^{-1}$). We see
    \begin{equation*}
        \F[P_NG(P_Kf,P_Lg)](n)=N^{-1}\sum_{j=1,2}\sum_{k+\ell=n}m_N^j(k,\ell)(k_j\ft{P_K f}(k)\cdot\ell)\ft{P_L g}(\ell)\;,
    \end{equation*}
    and so by Proposition \ref{prop:CM} and Bernstein's inequalities we see
    \begin{align*}
        \|P_NG(P_Kf,P_Lg)\|_{L^\infty} &\lesssim N^{\varepsilon} \|P_NG(P_Kw,P_Lw)\|_{L^{2/\varepsilon}} \\
        &\lesssim N^{\varepsilon-1}LK\|P_Kf\|_{L^\infty}\|P_Lg\|_{L^\infty} \\
        &\lesssim N^{\varepsilon-\beta}K^{1-\beta}\|f\|_{\CC^{\beta}} \|g\|_{\CC^{\beta}}\;.
    \end{align*}
    Next for $L\ll K\approx N$ we may use Bernstein's inequalities to directly bound \eqref{dyadicbilinear}
    \begin{align*}
        \|P_NG(P_Kf,P_Lg)\|_{L^\infty} &\lesssim L\|P_Kf\|_{L^\infty}\|P_Lg\|_{L^\infty} \\
        &\lesssim N^{-\beta} L^{1-\beta}\|f\|_{\CC^{\beta}} \|g\|_{\CC^{\beta}}\;.
    \end{align*}
    For the resonant case $K\approx L\gtrsim N$, we decompose \eqref{dyadicbilinear} as
    \begin{align}
        \F[P_NG(P_Kf,P_Lg)](n)&=i\sum_{k+\ell=n}\phi_N(n)(\ft{P_K f}(k)\cdot n)\ft{P_L g}(\ell)\label{Gresonantterm1} \\
        &\quad -i\sum_{k+\ell=n}\phi_N(n)|n|^{-\alpha}|\ell|^{\alpha}(\ft{P_K f}(k)\cdot n)\ft{P_L g}(\ell)\;,\label{Gresonantterm2}
    \end{align}
    where we have used that $f$ is divergence-free. From this we conclude that
    \begin{align*}
        \|\eqref{Gresonantterm1}\|_{L^\infty} &\lesssim N\|P_Kf\|_{L^\infty}\|P_Lg\|_{L^\infty} \\
        &\lesssim N L^{-\beta}K^{-\beta}\|f\|_{\CC^{\beta}} \|g\|_{\CC^{\beta}}\;,
    \end{align*}
   	and similarly
   	\begin{align*}
        \|\eqref{Gresonantterm2}\|_{L^\infty} &\lesssim N^{1-\alpha} L^{\alpha}\|P_Kf\|_{L^\infty}\|P_Lg\|_{L^\infty} \\
        &\lesssim N^{1-\alpha} L^{\alpha-\beta}K^{-\beta}\|f\|_{\CC^{\beta}} \|g\|_{\CC^{\beta}}\;.
    \end{align*}
    Putting this all together and writing $\hsum$ for sums restricted to dyadic integers, we have for $\beta>\alpha/2$
    \begin{equs}
        \|&P_NG(f,g)\|_{L^\infty} \\
        &\lesssim \biggl(\hsum_{K\ll L\approx N}+\hsum_{L\ll K\approx N}+\hsum_{K\approx L\gtrsim N}\biggr)\|P_NG(P_Kw,P_Lw)\|_{L^\infty} \\
        &\lesssim\biggl(\hsum_{K\ll L\approx N}N^{\varepsilon-\beta}K^{1-\beta}+\hsum_{L\ll K\approx N}N^{-\beta} L^{1-\beta}+\hsum_{K\approx L\gtrsim N}N^{1-\alpha} L^{\alpha-\beta}K^{-\beta}\biggr) \\
        &\quad\,\, \times \|f\|_{\CC^{\beta}} \|g\|_{\CC^{\beta}} \\
        &\lesssim (N^{\varepsilon-\beta+\max\{0,1-\beta\}}+N^{-\beta+\max\{0,1-\beta\}}+N^{1-2\beta})\|f\|_{\CC^{\beta}} \|g\|_{\CC^{\beta}}\;,
    \end{equs}
    which yields the result.
    \end{proof}
\begin{proof}[Proof of Theorems \ref{theo:main} and \ref{theo:mainGeneral}]
    Let $X^{x_0},Y^{x_0}$ be the solutions of \eqref{e:equvX} and \eqref{e:equvY} respectively, with initial data $x_0 \in \CC^{\al-\kappa}$. By Propositions~\ref{prop:apriori}, \ref{prop:asGWPv}, and~\ref{prop:GalRegularity} with $\beta = \alpha-\kappa$, we see that $X,Y$ satisfy the assumptions of Proposition \ref{prop:equiv}. In particular, for every $t \ge 0$, and every $x_0 \in \CC^{\al-\kappa}$, we have that $\Law(X^{x_0}_t) \sim \Law(Y^{x_0}_t)$. By Proposition~\ref{prop:equivalence}, we also have 
    $$ \rho \sim \Law(X^{x_0}_t) \,$$
    for every $x_0 \in \CC^{\al-\kappa}$ and $t > 0$.
    By averaging over $\mu_\al$, we obtain that for every $t > 0$,
    $$ \rho = \int \rho\, d\mu_\al(x_0) \sim \int \Law(Y^{x_0}_t)  \,d\mu_\al(x_0)\;. $$
    By Proposition \ref{prop:asGWPv}, $\mu_\al$ is an invariant measure for $Y$, so for every $t > 0$,
    $$ \int \Law(Y^{x_0}_t) \, d\mu_\al(x_0) = \mu_\al\;. $$
    In particular, 
    $ \rho \sim \mu_\al$ as claimed.
\end{proof}


\bibliographystyle{Martin}

\bibliography{references_2}

\begin{thebibliography}{BCCH21}
\def\myhref#1#2{\href{#2}{\nolinkurl{#1}}}

\bibitem[BCCH21]{GenRen}
\textsc{Y.~Bruned}, \textsc{A.~Chandra}, \textsc{I.~Chevyrev}, and
  \textsc{M.~Hairer}.
\newblock Renormalising {SPDE}s in regularity structures.
\newblock \emph{J. Eur. Math. Soc. (JEMS)} \textbf{23}, no.~3, (2021),
  869--947.
\newblock \myhref{arXiv:1711.10239}{https://arxiv.org/abs/1711.10239}.
\newblock \myhref{doi:10.4171/jems/1025}{https://doi.org/10.4171/jems/1025}.

\bibitem[BHZ19]{AlgRen}
\textsc{Y.~Bruned}, \textsc{M.~Hairer}, and \textsc{L.~Zambotti}.
\newblock Algebraic renormalisation of regularity structures.
\newblock \emph{Invent. Math.} \textbf{215}, no.~3, (2019), 1039--1156.
\newblock \myhref{arXiv:1610.08468}{https://arxiv.org/abs/1610.08468}.
\newblock
  \myhref{doi:10.1007/s00222-018-0841-x}{https://doi.org/10.1007/s00222-018-0841-x}.

\bibitem[BKL01]{MR1868991}
\textsc{J.~Bricmont}, \textsc{A.~Kupiainen}, and \textsc{R.~Lefevere}.
\newblock Ergodicity of the 2{D} {N}avier-{S}tokes equations with random
  forcing.
\newblock \emph{Comm. Math. Phys.} \textbf{224}, no.~1, (2001), 65--81.
\newblock Dedicated to Joel L. Lebowitz.
\newblock
  \myhref{doi:10.1007/s002200100510}{https://doi.org/10.1007/s002200100510}.

\bibitem[BOZ25]{Benyi24}
\textsc{A.~B\'enyi}, \textsc{T.~Oh}, and \textsc{T.~Zhao}.
\newblock Fractional {L}eibniz rule on the torus.
\newblock \emph{Proc. Amer. Math. Soc.} \textbf{153}, no.~1, (2025), 207--221.
\newblock \myhref{doi:10.1090/proc/17007}{https://doi.org/10.1090/proc/17007}.

\bibitem[CK19]{CK}
\textsc{D.~Cardona} and \textsc{V.~Kumar}.
\newblock {$L^p$}-boundedness and {$L^p$}-nuclearity of multilinear
  pseudo-differential operators on {${\Bbb Z}^n$} and the torus {${\Bbb T}^n$}.
\newblock \emph{J. Fourier Anal. Appl.} \textbf{25}, no.~6, (2019), 2973--3017.
\newblock
  \myhref{doi:10.1007/s00041-019-09689-7}{https://doi.org/10.1007/s00041-019-09689-7}.

\bibitem[CZH21]{MR4244269}
\textsc{M.~Coti~Zelati} and \textsc{M.~Hairer}.
\newblock A noise-induced transition in the {L}orenz system.
\newblock \emph{Comm. Math. Phys.} \textbf{383}, no.~3, (2021), 2243--2274.
\newblock \myhref{arXiv:2004.12815}{https://arxiv.org/abs/2004.12815}.
\newblock
  \myhref{doi:10.1007/s00220-021-04000-6}{https://doi.org/10.1007/s00220-021-04000-6}.

\bibitem[DPD02]{DPD}
\textsc{G.~Da~Prato} and \textsc{A.~Debussche}.
\newblock Two-dimensional {N}avier-{S}tokes equations driven by a space-time
  white noise.
\newblock \emph{J. Funct. Anal.} \textbf{196}, no.~1, (2002), 180--210.
\newblock
  \myhref{doi:10.1006/jfan.2002.3919}{https://doi.org/10.1006/jfan.2002.3919}.

\bibitem[DPZ96]{DPZ2}
\textsc{G.~Da~Prato} and \textsc{J.~Zabczyk}.
\newblock \emph{Ergodicity for infinite-dimensional systems}, vol. 229 of
  \emph{London Mathematical Society Lecture Note Series}.
\newblock Cambridge University Press, Cambridge, 1996,  xii+339.
\newblock
  \myhref{doi:10.1017/CBO9780511662829}{https://doi.org/10.1017/CBO9780511662829}.

\bibitem[DPZ14]{Da_Prato_Zabczyk_2014}
\textsc{G.~Da~Prato} and \textsc{J.~Zabczyk}.
\newblock \emph{Stochastic equations in infinite dimensions}, vol. 152 of
  \emph{Encyclopedia of Mathematics and its Applications}.
\newblock Cambridge University Press, Cambridge, second ed., 2014,  xviii+493.
\newblock
  \myhref{doi:10.1017/CBO9781107295513}{https://doi.org/10.1017/CBO9781107295513}.

\bibitem[EMS01]{Ee-Mattingly-Sinai}
\textsc{W.~E}, \textsc{J.~C. Mattingly}, and \textsc{Y.~Sinai}.
\newblock Gibbsian dynamics and ergodicity for the stochastically forced
  {N}avier-{S}tokes equation.
\newblock \emph{Comm. Math. Phys.} \textbf{224}, no.~1, (2001), 83--106.
\newblock Dedicated to Joel L. Lebowitz.
\newblock
  \myhref{doi:10.1007/s002201224083}{https://doi.org/10.1007/s002201224083}.

\bibitem[Fer70]{Fernique}
\textsc{X.~Fernique}.
\newblock Int\'egrabilit\'e{} des vecteurs {G}aussiens.
\newblock \emph{C. R. Acad. Sci. Paris S\'er. A-B} \textbf{270}, (1970),
  A1698--A1699.

\bibitem[Fer97]{Benedetta}
\textsc{B.~Ferrario}.
\newblock Ergodic results for stochastic {N}avier-{S}tokes equation.
\newblock \emph{Stochastics Stochastics Rep.} \textbf{60}, no. 3-4, (1997),
  271--288.
\newblock
  \myhref{doi:10.1080/17442509708834110}{https://doi.org/10.1080/17442509708834110}.

\bibitem[Fla94]{Franco}
\textsc{F.~Flandoli}.
\newblock Dissipativity and invariant measures for stochastic {N}avier-{S}tokes
  equations.
\newblock \emph{NoDEA Nonlinear Differential Equations Appl.} \textbf{1},
  no.~4, (1994), 403--423.
\newblock \myhref{doi:10.1007/BF01194988}{https://doi.org/10.1007/BF01194988}.

\bibitem[FM95]{ErgodNS}
\textsc{F.~Flandoli} and \textsc{B.~Maslowski}.
\newblock Ergodicity of the {$2$}-{D} {N}avier-{S}tokes equation under random
  perturbations.
\newblock \emph{Comm. Math. Phys.} \textbf{172}, no.~1, (1995), 119--141.
\newblock \urlprefix\url{https://projecteuclid.org/euclid.cmp/1104273961}.

\bibitem[FM02]{Fry02}
\textsc{R.~Fry} and \textsc{S.~McManus}.
\newblock Smooth bump functions and the geometry of {B}anach spaces: a brief
  survey.
\newblock \emph{Expo. Math.} \textbf{20}, no.~2, (2002), 143--183.
\newblock
  \myhref{doi:10.1016/S0723-0869(02)80017-2}{https://doi.org/10.1016/S0723-0869(02)80017-2}.

\bibitem[Hai14]{RegStruc}
\textsc{M.~Hairer}.
\newblock A theory of regularity structures.
\newblock \emph{Invent. Math.} \textbf{198}, no.~2, (2014), 269--504.
\newblock \myhref{arXiv:1303.5113}{https://arxiv.org/abs/1303.5113}.
\newblock
  \myhref{doi:10.1007/s00222-014-0505-4}{https://doi.org/10.1007/s00222-014-0505-4}.

\bibitem[Hai23]{SPDENotes}
\textsc{M.~Hairer}.
\newblock An introduction to stochastic {PDE}s.
\newblock \emph{arXiv preprint} (2023).
\newblock \myhref{arXiv:0907.4178}{https://arxiv.org/abs/0907.4178}.

\bibitem[HKN24]{Singular}
\textsc{M.~Hairer}, \textsc{S.~Kusuoka}, and \textsc{H.~Nagoji}.
\newblock Singularity of solutions to singular {SPDE}s.
\newblock \emph{arXiv preprint} (2024).
\newblock \myhref{arXiv:2409.10037}{https://arxiv.org/abs/2409.10037}.

\bibitem[HM06]{MR2259251}
\textsc{M.~Hairer} and \textsc{J.~C. Mattingly}.
\newblock Ergodicity of the 2{D} {N}avier-{S}tokes equations with degenerate
  stochastic forcing.
\newblock \emph{Ann. of Math. (2)} \textbf{164}, no.~3, (2006), 993--1032.
\newblock
  \myhref{doi:10.4007/annals.2006.164.993}{https://doi.org/10.4007/annals.2006.164.993}.

\bibitem[HM08]{MR2478676}
\textsc{M.~Hairer} and \textsc{J.~C. Mattingly}.
\newblock Spectral gaps in {W}asserstein distances and the 2{D} stochastic
  {N}avier-{S}tokes equations.
\newblock \emph{Ann. Probab.} \textbf{36}, no.~6, (2008), 2050--2091.
\newblock \myhref{doi:10.1214/08-AOP392}{https://doi.org/10.1214/08-AOP392}.

\bibitem[HM18]{SFGeneral}
\textsc{M.~Hairer} and \textsc{J.~Mattingly}.
\newblock The strong {F}eller property for singular stochastic {PDE}s.
\newblock \emph{Ann. Inst. Henri Poincar\'e{} Probab. Stat.} \textbf{54},
  no.~3, (2018), 1314--1340.
\newblock \myhref{doi:10.1214/17-AIHP840}{https://doi.org/10.1214/17-AIHP840}.

\bibitem[HR24]{TommasoNS}
\textsc{M.~Hairer} and \textsc{T.~Rosati}.
\newblock Global existence for perturbations of the 2{D} stochastic
  {N}avier-{S}tokes equations with space-time white noise.
\newblock \emph{Ann. PDE} \textbf{10}, no.~1, (2024), Paper No. 3, 46.
\newblock \myhref{arXiv:2301.11059}{https://arxiv.org/abs/2301.11059}.
\newblock
  \myhref{doi:10.1007/s40818-023-00165-6}{https://doi.org/10.1007/s40818-023-00165-6}.

\bibitem[HS24]{Rhys}
\textsc{M.~Hairer} and \textsc{R.~Steele}.
\newblock The {BPHZ} theorem for regularity structures via the spectral gap
  inequality.
\newblock \emph{Arch. Ration. Mech. Anal.} \textbf{248}, no.~1, (2024), Paper
  No. 9, 81.
\newblock \myhref{arXiv:2301.10081}{https://arxiv.org/abs/2301.10081}.
\newblock
  \myhref{doi:10.1007/s00205-023-01946-w}{https://doi.org/10.1007/s00205-023-01946-w}.

\bibitem[HZ25]{WenhaoNS}
\textsc{M.~Hairer} and \textsc{W.~Zhao}.
\newblock Ergodicity of 2{D} singular stochastic {N}avier--{S}tokes equations.
\newblock \emph{Probab. Math. Phys.} \textbf{6}, no.~3, (2025), 777--818.
\newblock \myhref{arXiv:2411.03482}{https://arxiv.org/abs/2411.03482}.
\newblock
  \myhref{doi:10.2140/pmp.2025.6.777}{https://doi.org/10.2140/pmp.2025.6.777}.

\bibitem[KPV93]{KPV}
\textsc{C.~E. Kenig}, \textsc{G.~Ponce}, and \textsc{L.~Vega}.
\newblock Well-posedness and scattering results for the generalized
  {K}orteweg-de {V}ries equation via the contraction principle.
\newblock \emph{Comm. Pure Appl. Math.} \textbf{46}, no.~4, (1993), 527--620.
\newblock
  \myhref{doi:10.1002/cpa.3160460405}{https://doi.org/10.1002/cpa.3160460405}.

\bibitem[KS00]{MR1785459}
\textsc{S.~Kuksin} and \textsc{A.~Shirikyan}.
\newblock Stochastic dissipative {PDE}s and {G}ibbs measures.
\newblock \emph{Comm. Math. Phys.} \textbf{213}, no.~2, (2000), 291--330.
\newblock
  \myhref{doi:10.1007/s002200000237}{https://doi.org/10.1007/s002200000237}.

\bibitem[MRS22]{MRS22}
\textsc{J.~C. Mattingly}, \textsc{M.~Romito}, and \textsc{L.~Su}.
\newblock The {G}aussian structure of the singular stochastic {B}urgers
  equation.
\newblock \emph{Forum Math. Sigma} \textbf{10}, (2022), Paper No. e75, 47.
\newblock
  \myhref{doi:10.1017/fms.2022.64}{https://doi.org/10.1017/fms.2022.64}.

\bibitem[MS05]{Mattingly2005-aj}
\textsc{J.~C. Mattingly} and \textsc{T.~M. Suidan}.
\newblock The small scales of the stochastic {N}avier-{S}tokes equations under
  rough forcing.
\newblock \emph{J. Stat. Phys.} \textbf{118}, no. 1-2, (2005), 343--364.
\newblock
  \myhref{doi:10.1007/s10955-004-8787-3}{https://doi.org/10.1007/s10955-004-8787-3}.

\bibitem[MS08]{Suidan2}
\textsc{J.~C. Mattingly} and \textsc{T.~M. Suidan}.
\newblock Transition measures for the stochastic {B}urgers equation.
\newblock In \emph{Integrable systems and random matrices}, vol. 458 of
  \emph{Contemp. Math.},  409--418. Amer. Math. Soc., Providence, RI, 2008.
\newblock
  \myhref{doi:10.1090/conm/458/08950}{https://doi.org/10.1090/conm/458/08950}.

\bibitem[MZ08]{ezz08}
\textsc{R.~May} and \textsc{E.~Zahrouni}.
\newblock Global existence of solutions for subcritical quasi-geostrophic
  equations.
\newblock \emph{Commun. Pure Appl. Anal.} \textbf{7}, no.~5, (2008),
  1179--1191.
\newblock
  \myhref{doi:10.3934/cpaa.2008.7.1179}{https://doi.org/10.3934/cpaa.2008.7.1179}.

\bibitem[Wat10]{Girsanov2}
\textsc{A.~Watkins}.
\newblock On absolute continuity for stochastic partial differential equations
  and an averaging principle for a queueing network, 2010.
\newblock \urlprefix\url{https://hdl.handle.net/10161/3048}.
\newblock PhD thesis, Duke University.

\bibitem[Wu01]{Wu01}
\textsc{J.~Wu}.
\newblock Dissipative quasi-geostrophic equations with {$L^p$} data.
\newblock \emph{Electron. J. Differential Equations} (2001), No. 56, 13.
\newblock \urlprefix\url{https://eudml.org/doc/121853}.

\end{thebibliography}
\end{document}